\documentclass{article}
\usepackage{amsmath,amssymb,amsthm}

\usepackage[dvipdfmx]{graphicx}
\graphicspath{{./figure/}}

\usepackage{url}
\usepackage{color}
\usepackage{here}
\usepackage{hyperref}

\begin{document}

\theoremstyle{plain}
 \newtheorem{thm}{Theorem}[section]
 \newtheorem{lem}[thm]{Lemma}
 \newtheorem{prop}[thm]{Proposition}
 \newtheorem{cor}[thm]{Corollary}
 \newtheorem{cl}[thm]{Claim}
 \newtheorem{conj}[thm]{Conjecture}
 \newtheorem{ques}[thm]{Question}
\theoremstyle{definition}
 \newtheorem{dfn}[thm]{Definition}
\theoremstyle{remark}
 \newtheorem{rem}[thm]{Remark}

\newcommand{\fig}[3][width=12cm]{
\begin{figure}[H]
 \centering 
 \includegraphics[#1,clip]{#2} 
 \caption{#3} 
\label{fig:#2}
\end{figure}}


\newcommand{\blank}{\vspace{0.5\baselineskip}}

\newcommand{\vol}{\mathrm{vol}}

\newcommand{\Whi}{\mathit{Whi}}
\newcommand{\Whip}{\mathit{Whi^{\prime}}}
\newcommand{\Whih}{\widehat{\mathit{Whi}}}
\newcommand{\Whiph}{\widehat{\mathit{Whi^{\prime}}}}
\newcommand{\Bor}{\mathit{Bor}_{6}}
\newcommand{\Mag}{\mathit{Mag}_{4}}
\newcommand{\Tet}{\mathit{Tet}_{8}}
\newcommand{\Pen}{\mathit{Pen}_{10}}
\newcommand{\Oct}{\mathit{Oct}_{8}}
\newcommand{\Teth}{\widehat{\mathit{Tet}}_{2}}
\newcommand{\Penh}{\widehat{\mathit{Pen}}_{4}}
\newcommand{\Octh}{\widehat{\mathit{Oct}}_{4}}

\title{UNIONS OF 3-PUNCTURED SPHERES IN HYPERBOLIC 3-MANIFOLDS}
\author{KEN'ICHI YOSHIDA}
\date{}

\maketitle

\begin{abstract}
We classify the topological types for the unions 
of the totally geodesic 3-punctured spheres 
in orientable hyperbolic 3-manifolds. 
General types of the unions appear in various hyperbolic 3-manifolds. 
Each of the special types of the unions appears 
only in a single hyperbolic 3-manifold 
or Dehn fillings of a single hyperbolic 3-manifold. 
Furthermore, we investigate bounds of the moduli of adjacent cusps 
for the union of linearly placed 3-punctured spheres. 
\end{abstract}

\section{Introduction}
\label{section:intro}

In this paper, 
we consider totally geodesic 3-punctured spheres 
in orientable hyperbolic 3-manifolds. 
The $\epsilon$-thick part of an orientable hyperbolic 3-manifold $M$ is 
its submanifold $M_{[\epsilon, \infty)}$ 
such that the open ball of radius $\epsilon$ 
centered at any $x \in M_{[\epsilon, \infty)}$ is embedded in $M$. 
We call it simply the thick part 
after fixing $\epsilon$ to be at most the Margulis constant for $\mathbb{H}^{3}$. 
Then the thin part (i.e. the complement of the thick part) 
is the disjoint union of tubes and cusp neighborhoods. 
A tube is a regular neighbourhood of a closed geodesic of length less than $2\epsilon$. 
By removing the cusp neighborhoods from $M$, we obtain a 3-manifold $M_{0}$. 
Then the interior of $M_{0}$ is homeomorphic to $M$, 
and a boundary component of $M_{0}$ is a torus or annulus, called a cusp. 
For convenience, 
we ignore the distinction between $M$ and $M_{0}$. 
Thus an orientable hyperbolic 3-manifold of finite volume 
is regarded as a compact 3-manifold with (possibly empty) boundary consisting of torus cusps.

The upper half-space model gives the identifications of 
the ideal boundary 
$\partial \mathbb{H}^{3} \cong \widehat{\mathbb{C}} 
= \mathbb{C} \cup \{ \infty \}$ 
and the group of the orientation-preserving isometries 
$\mathrm{Isom}^{+}(\mathbb{H}^{3}) \cong \mathrm{PSL}(2,\mathbb{C})$. 
A torus cusp neighborhood  is isometric to 
a neighborhood of the image of $\infty$ in $\mathbb{H}^{3}/G$, 
where $G \cong \mathbb{Z}^{2}$ is a group 
that consists of parabolic elements fixing $\infty$. 
Thus a torus cusp admits the natural Euclidean structure 
up to scaling.

A 3-punctured sphere is obtained by removing three points from the 2-sphere, 
but we often regard it as 
a compact orientable surface of genus zero with three boundary components. 
We always assume that 
the boundary of a 3-punctured sphere in a hyperbolic 3-manifold 
is contained in the cusps. 
Adams~\cite{adams1985thrice} showed that 
an essential 3-punctured sphere in an orientable hyperbolic 3-manifold 
is isotopic to a totally geodesic one. 
A totally geodesic 3-punctured sphere is 
isometric to the double of an ideal triangle 
in the hyperbolic plane $\mathbb{H}^{2}$. 
Moreover, the hyperbolic structure of a 3-punctured sphere 
is unique up to isometry. 
After we cut a hyperbolic 3-manifold 
along a totally geodesic 3-punctured sphere,  
we can glue it again by an isometry along the new boundary. 
Since there are six orientation-preserving isometries 
of a totally geodesic 3-punctured sphere, 
we can construct hyperbolic 3-manifolds with a common volume.

We remark on an immersed 3-punctured sphere. 
The existence of a non-embedded immersed 3-punctured sphere 
almost determines the ambient hyperbolic 3-manifold. 

\begin{thm}[Agol~\cite{agol2009pants}]
\label{thm:immersed}
Let $\Sigma$ be an immersed 3-punctured sphere 
in an orientable hyperbolic 3-manifold $M$. 
Suppose that $\Sigma$ is not homotopic to an embedded one. 
Then $M$ is obtained by a (possibly empty) Dehn filling 
on a cusp of the Whitehead link complement. 
Furthermore, $\Sigma$ is homotopic to a totally geodesic 3-punctured sphere 
immersed in $M$ as shown in Figure~\ref{fig:u3ps-nonemb}. 
\end{thm}

\fig[width=5cm]{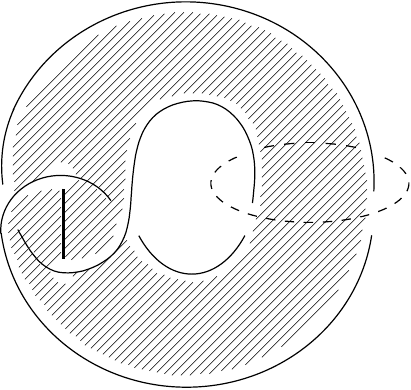}{Non-embedded 3-punctured sphere}

From now on, we consider embedded totally geodesic 3-punctured spheres. 
If all the 3-punctured spheres in a hyperbolic 3-manifold are disjoint, 
we can standardly decompose the 3-manifold along the 3-punctured spheres. 
However, 3-punctured spheres may intersect. 
Thus we consider all the 3-punctured spheres. 
Although the unions of 3-punctured spheres may be complicated, 
we can classify them. 
The JSJ decomposition of an irreducible orientable 3-manifold 
gives atoroidal pieces and Seifert pieces, 
and then every essential torus in the 3-manifold can be isotoped 
into a Seifert piece. 
Theorem~\ref{thm:main} can be regarded as 
an analog of the classification of the Seifert 3-manifolds.

\begin{thm}
\label{thm:main}
Let $M$ be an orientable hyperbolic 3-manifold. 
Suppose that $X$ is a connected component of the union 
of all the (embedded) totally geodesic 3-punctured spheres in $M$. 
Let $N(X)$ be a regular neighborhood of $X$. 
By abuse of terminology, 
we refer to the topological type of the pair $(N(X), X)$ as the type of $X$. 
If $X$ consists of finitely many 3-punctured spheres, 
then $X$ is one of the following types: 
\begin{itemize}
\item (general types) 

$A_{n} (n \geq 1), B_{2n} (n \geq 1),  T_{3}, T_{4}.$
\item (types determining the manifolds) 

$\Whi_{2n} (n \geq 2), \Whip_{4n} (n \geq 2), \Bor, \Mag, \Tet, \Pen, \Oct.$
\item (types almost determining the manifolds) 

$\Whih_{n} (n \geq 2), \Whiph_{2n} (n \geq 1), \Teth, \Penh, \Octh.$
\end{itemize}
The indices indicate the numbers of 3-punctured spheres. 
\end{thm}

\fig[width=10cm]{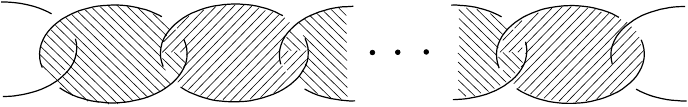}{$A_{n}$}
\fig[width=10cm]{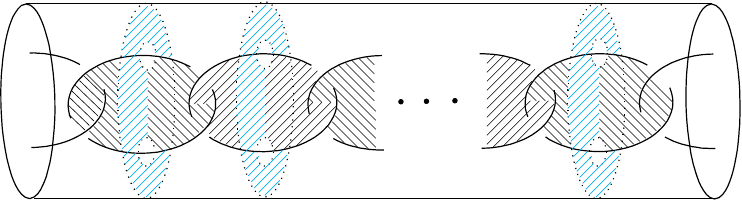}{$B_{2n}$}
\fig[width=8cm]{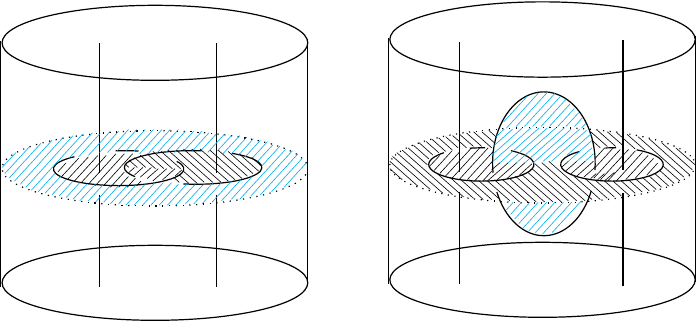}{$T_{3}$ and $T_{4}$}
\fig[width=12cm]{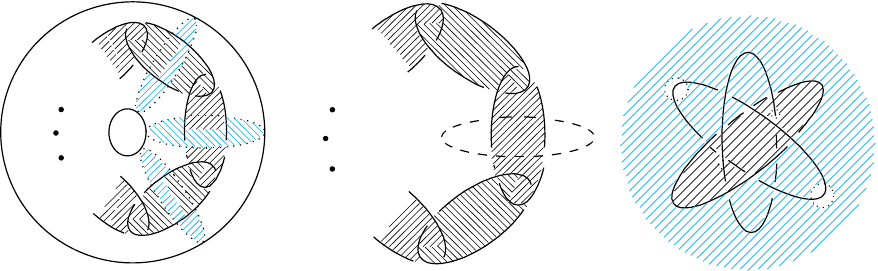}{$\mathit{Whi^{(\prime)}}_{2n}, \widehat{\mathit{Whi^{(\prime)}}}_{n}$, and $\Bor$}
\fig[width=12cm]{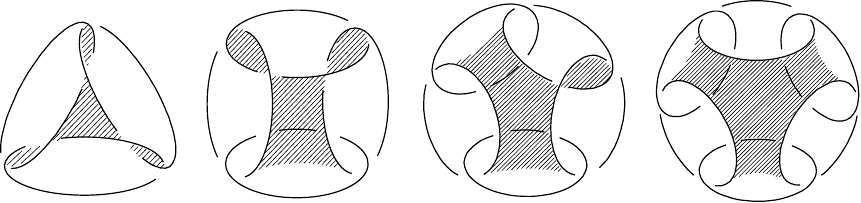}{$\Mag, \Tet, \Pen$, and $\Oct$}
\fig[width=12cm]{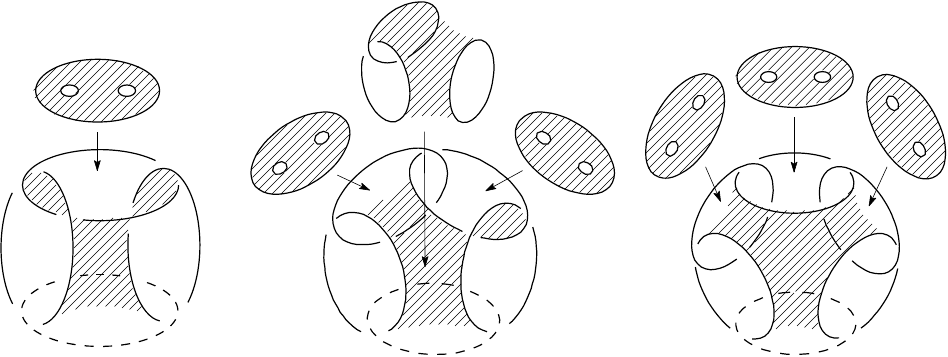}{$\Teth, \Penh$, and $\Octh$}

\fig[width=12cm]{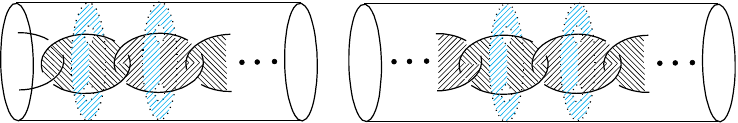}{$B_{\infty}$ and $\Whi_{\infty}$}

\begin{thm}
\label{thm:infinite}

Let $M$ be an orientable hyperbolic 3-manifold. 
Suppose that $X$ is a connected component of the union 
of all the totally geodesic 3-punctured spheres in $M$. 
If $X$ consists of infinitely many 3-punctured spheres, 
then $X$ is the type $B_{\infty}$ or $\Whi_{\infty}$. 
\end{thm}

Careful descriptions of these types will be given in Section~\ref{section:description}. 
Now we explain them briefly. 
The general types appear in various manifolds. 
For any finite multiset of general types, 
there are infinitely many hyperbolic 3-manifolds 
containing 3-punctured spheres of such types. 
When the type $B_{2n}, T_{3}$, or $T_{4}$ appears, 
there are additional isolated 3-punctured spheres, 
which are contained in the boundary of 3-manifolds 
shown in Figures~\ref{fig:u3ps-b2n} and \ref{fig:u3ps-t34}. 
In contrast, each of the determining types 
appears only in a certain special manifold. 
Not all the 3-punctured spheres are shown 
in Figures~\ref{fig:u3ps-whibor} and \ref{fig:u3ps-magtetpenoct}, 
because there are too many 3-punctured spheres. 
The almost determining types 
appear only in manifolds obtained by Dehn fillings 
of such special manifolds. 
The dashed circles in Figures~\ref{fig:u3ps-whibor} and \ref{fig:u3ps-tetpenocthat} 
indicate filled cusps. 
For an (almost) determining type, the ambient 3-manifold has finite volume. 
For each $n \geq 2$, 
the unions of 3-punctured spheres of the types $\Whi_{4n}$ and $\Whip_{4n}$ 
have a common topology as topological spaces, 
but they are distinguished by their neighborhoods. 
The same argument holds for $\Whih_{2n}$ and $\Whiph_{2n}$.

For $3 \leq n \leq 6$, let $\mathbb{M}_{n}$ denote 
the minimally twisted hyperbolic $n$-chain link complement 
as shown in Figure~\ref{fig:u3ps-magtetpenoct}. 
The 3-punctured spheres of the types 
$\Mag, \Tet, \Pen$, and $\Oct$ 
are respectively contained only in the manifolds
$\mathbb{M}_{3}, \mathbb{M}_{4}, \mathbb{M}_{5}$, and $\mathbb{M}_{6}$. 
These manifolds are quite special 
for several reasons. 
Agol~\cite{agol2010minimal} conjectured that they have the smallest volume 
of the $n$-cusped orientable hyperbolic 3-manifolds. 
This conjecture was proven for $\mathbb{M}_{4}$ 
by the author~\cite{yoshida2013minimal}.

The manifold $\mathbb{M}_{3}$ is called the magic manifold 
by Gordon and Wu~\cite{gordon1999toroidal}. 
It is known that many interesting examples are obtained 
by Dehn fillings of $\mathbb{M}_{3}$. 
For example, 
Kin and Takasawa~\cite{kin2011pseudo} showed that 
some mapping tori of punctured disk fibers with small entropy 
appear as Dehn fillings of $\mathbb{M}_{3}$. 
The manifold $\mathbb{M}_{5}$ is known to 
give most of the census manifolds~\cite{callahan1999census} 
by Dehn fillings. 
Martelli, Petronio, and Roukema~\cite{martelli2014exceptional} 
classified the non-hyperbolic manifolds 
obtained by Dehn fillings of $\mathbb{M}_{5}$. 
Kolpakov and Martelli~\cite{kolpakov2013hyperbolic} constructed 
the first examples of finite volume hyperbolic 4-manifolds 
with exactly one cusp 
by using $\mathbb{M}_{6}$. 
Baker~\cite{baker2002all} showed that 
every link in $S^{3}$ is a sublink of a link 
whose complement is a covering of $\mathbb{M}_{6}$. 
He then used the link in Figure~\ref{fig:u3ps-m6var} 
to construct coverings of $\mathbb{M}_{6}$ efficiently. 
We can show that the complement of this link is 
homeomorphic to $\mathbb{M}_{6}$ 
by finding eight 3-punctured spheres of the type $\Oct$ 
(see also \cite[Lemma 5.9]{kin2018braids}).

Moreover, 
the manifolds $\mathbb{M}_{3}, \mathbb{M}_{4}, \mathbb{M}_{5}$, 
and $\mathbb{M}_{6}$ 
are arithmetic hyperbolic 3-manifolds. 
A cusped finite volume hyperbolic 3-manifold is arithmetic 
if and only if 
its fundamental group is commensurable to 
a Bianchi group $\mathrm{PSL}(2,\mathcal{O}_{d})$, 
where $\mathcal{O}_{d}$ is the ring of integers 
of the imaginary quadratic field $\mathbb{Q}(\sqrt{-d})$ 
(see \cite[Ch. 8]{maclachlan2013arithmetic} for more details). 
Thurston~\cite[Ch. 6]{thurston1978geometry} gave 
an explicit representation of $\pi_{1}(\mathbb{M}_{3})$ 
as a subgroup of $\mathrm{PSL}(2,\mathcal{O}_{7})$. 
Baker~\cite{baker2002all} gave 
an explicit representation of $\pi_{1}(\mathbb{M}_{4})$ 
as a subgroup of $\mathrm{PSL}(2,\mathcal{O}_{1})$, 
and showed that 
$\mathbb{M}_{6}$ is a double covering of $\mathbb{M}_{4}$. 
The fundamental group of the Whitehead link complement 
is also commensurable to $\mathrm{PSL}(2,\mathcal{O}_{1})$. 
Hatcher~\cite{hatcher1983hyperbolic} showed that 
the fundamental group of a hyperbolic 3-manifold 
obtained from regular ideal tetrahedra 
(resp. regular ideal octahedra) 
is commensurable to 
$\mathrm{PSL}(2,\mathcal{O}_{3})$ 
(resp. $\mathrm{PSL}(2,\mathcal{O}_{1})$). 
As we will describe in Section~\ref{section:description}, 
the manifold $\mathbb{M}_{5}$ 
can be decomposed into ten regular ideal tetrahedra. 
Hence $\pi_{1}(\mathbb{M}_{5})$ is commensurable to 
$\mathrm{PSL}(2,\mathcal{O}_{3})$. 
In addition, Kin and Rolfsen~\cite{kin2018braids} studied 
bi-orderability of their fundamental groups. 

\fig[width=6cm]{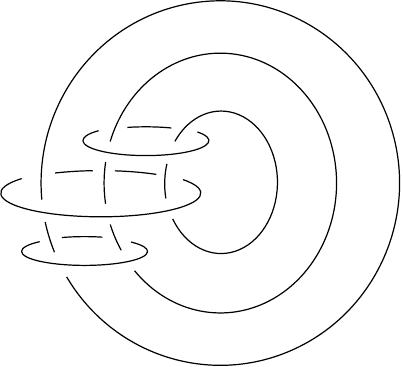}{A link whose complement is $\mathbb{M}_{6}$}

Eudave-Mu\~{n}oz and Ozawa~\cite{eudave2019characterization} 
characterized non-hyperbolic 3-component links in the 3-sphere 
whose complements contain essential 3-punctured spheres 
with non-integral boundary slopes. 
Moreover, they conjectured that a hyperbolic link complement does not contain 
an essential $n$-punctured sphere 
with non-meridional and non-integral boundary slope. 
Since our result does not concern embeddings of hyperbolic 3-manifolds 
in the 3-sphere, 
we do not solve this conjecture for 3-punctured spheres. 
Nevertheless, our result might be helpful to approach this conjecture.

\section{Description of the types of the unions of 3-punctured spheres}
\label{section:description}

In this section 
we describe the types of the unions of 3-punctured spheres 
in Theorems~\ref{thm:main} and \ref{thm:infinite}. 
This section concerns the existence of the 3-punctured spheres. 
In Section~\ref{section:proof} we will show that 
each special manifold has no other 3-punctured spheres 
than the described ones. 

We first introduce manifolds containing the 3-punctured spheres 
of some types. 
Let $\mathbb{W}_{n}$ denote the manifold 
as shown in the left of Figure~\ref{fig:u3ps-whibor} 
that is an $n$-sheeted cyclic cover of the Whitehead link complement. 
Let $\mathbb{W}^{\prime}_{n}$ denote the manifold 
obtained by a half twist 
along a blue 3-punctured sphere of $\mathbb{W}_{n}$, 
which can be also shown in the left of Figure~\ref{fig:u3ps-whibor}. 
The manifolds $\mathbb{W}_{n}$ and $\mathbb{W}^{\prime}_{n}$
are homeomorphic to certain link complements. 
For odd $n$, the manifold $\mathbb{W}^{\prime}_{n}$ is homeomorphic 
to $\mathbb{W}_{n}$ by reversing orientation. 
Kaiser, Purcell and Rollins~\cite{kaiser2012volumes} 
described more details on these manifolds. 
Note that our notations are different from theirs. 
The manifolds 
$\mathbb{W}_{2n-1}, \mathbb{W}_{4n-2}, \mathbb{W}_{4n}, 
\mathbb{W}^{\prime}_{4n-2}$, and $\mathbb{W}^{\prime}_{4n}$
are respectively homeomorphic to 
$\widehat{W}_{2n-1}, \overline{W}_{4n-2}, \widehat{W}_{4n}, 
\widehat{W}_{4n-2}$, and $\overline{W}_{4n}$ 
in \cite{kaiser2012volumes}. 

It is well known that 
the Whitehead link complement $\mathbb{W}_{1}$ can be 
decomposed into a regular ideal octahedron. 
Hence 
$\vol(\mathbb{W}_{n}) = \vol(\mathbb{W}^{\prime}_{n}) = nV_{oct}$, 
where $V_{oct} = 3.6638...$ is the volume of a regular ideal octahedron. 
At present, 
this is the smallest known volume 
of the $(n+1)$-cusped orientable hyperbolic 3-manifolds for $n \geq 10$. 
The manifold $\mathbb{W}^{\prime}_{2}$ is the Borromean rings complement. 

We recall that 
the manifolds $\mathbb{M}_{3}, \dots , \mathbb{M}_{6}$ are 
the minimally twisted hyperbolic $n$-chain link complements 
for $n=3, \dots , 6$ 
as shown in Figure~\ref{fig:u3ps-magtetpenoct}. 

For $n \geq 1$, let $\mathbb{B}_{n}$ denote 
the hyperbolic 3-manifold with totally geodesic boundary 
obtained by cutting $\mathbb{W}_{n}$ 
along a blue 3-punctured sphere 
shown in the left of Figure~\ref{fig:u3ps-whibor}. 
The manifold $\mathbb{B}_{n}$ is shown in Figure~\ref{fig:u3ps-b2n}. 
The manifold $\mathbb{B}_{1}$ can be decomposed 
into a regular ideal octahedron 
as shown in Figure~\ref{fig:u3ps-b1decomp}, 
where faces $X$ and $X^{\prime}$ are glued 
so that the orientations of edges match. 
We will use this decomposition in Section~\ref{section:parameter}. 

Similarly, 
let $\mathbb{T}_{3}$ and $\mathbb{T}_{4}$ denote 
the hyperbolic 3-manifolds with totally geodesic boundary 
respectively obtained by cutting $\mathbb{M}_{5}$ and $\mathbb{M}_{6}$ 
along a 3-punctured sphere. 
The manifolds $\mathbb{T}_{3}$ and $\mathbb{T}_{4}$ are shown in Figure~\ref{fig:u3ps-t34}. 

\fig[width=12cm]{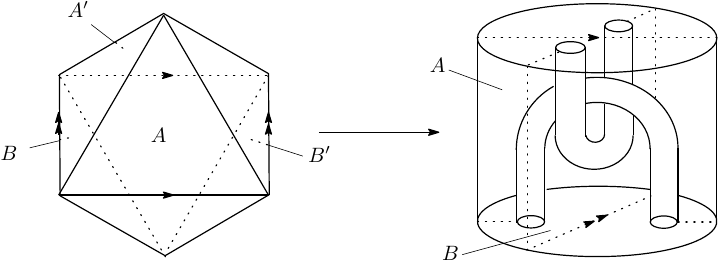}{Gluing of a regular ideal octahedron for $\mathbb{B}_{1}$}

\subsection*{$A_{n}$}

The 3-punctured spheres of the type $A_{n}$ are placed linearly, 
and can be regarded as the most general types. 
We consider an isolated 3-punctured sphere as the type $A_{1}$. 
For example, 
an appropriate Dehn filling of $\mathbb{W}_{n+3}$ gives a manifold 
with 3-punctured spheres of the type $A_{n}$.

\subsection*{$B_{2n}, T_{3}, T_{4}$}

The manifolds $\mathbb{B}_{n+1}, \mathbb{T}_{3}$, and $\mathbb{T}_{4}$ 
respectively contain 
3-punctured spheres of the types $B_{2n}, T_{3}$, and $T_{4}$. 
If there are 3-punctured spheres 
of the type $B_{2n}, T_{3}$, or $T_{4}$, 
there are two (possibly identical) 
isolated 3-punctured spheres that correspond to 
the boundary of $\mathbb{T}_{3}, \mathbb{T}_{4}$, or $\mathbb{B}_{n+1}$. 

The 3-punctured spheres of the type $B_{2n}$ 
consist of $n$ 3-punctured spheres of the type $A_{n}$ 
and $n$ more 3-punctured spheres (shown in blue) which intersect the former ones. 
The 3-punctured spheres of the type $T_{3}$ 
intersect at a common geodesic. 
The type $T_{3}$ is symmetric. 
In other words, 
for any pair of 3-punctured spheres in $\mathbb{T}_{3}$, 
there is an isometry of $\mathbb{T}_{3}$
that maps one of the pair to the other. 
A blue 3-punctured sphere in the type $T_{4}$ intersects 
the three other 3-punctured spheres. 
The latter three 3-punctured spheres are symmetric. 

For each of the types $B_{2n}, T_{3}$, and $T_{4}$, 
the metric of neighborhood of the union is uniquely determined 
in $\mathbb{B}_{n+1}, \mathbb{T}_{3}$, or $\mathbb{T}_{4}$. 
In contrast, the metric of neighborhood of the union of the type $A_{n}$ for $n \geq 2$ 
depends on the modulus of an adjacent torus cusp. 
We will consider this modulus in Section~\ref{section:parameter}.

\subsection*{$\widehat{\mathit{Whi^{(\prime)}}}_{n}$}

If 3-punctured spheres are placed cyclically, 
their union is the type $\Whih_{n}$ or $\Whiph_{2n}$. 
The two types $\Whih_{2n}$ and $\Whiph_{2n}$ 
are distinguished by neighborhoods. 

Suppose that a hyperbolic 3-manifold $M$ contains 3-punctured spheres of the type $\Whih_{n}$.  Then the union of the $n$ 3-punctured spheres with the $n$ adjacent torus cusps 
has a regular neighborhood homeomorphic to the manifold $\mathbb{W}_{n}$. 
The ambient 3-manifold $M$ is obtained by a Dehn filling on a cusp of $\mathbb{W}_{n}$ 
since it is atoroidal. 
The same argument holds for $\Whiph_{2n}$ in $\mathbb{W}^{\prime}_{2n}$. 
In fact, such a surgered hyperbolic 3-manifold 
except $\mathbb{M}_{3}, \dots , \mathbb{M}_{6}$ 
contains no more 3-punctured spheres. 

We remark that the Whitehead link complement $\mathbb{W}_{1}$ 
has two embedded 3-punctured spheres of the type $\Whiph_{2}$.

\subsection*{$\mathit{Whi^{(\prime)}}_{2n}$}

The 3-punctured spheres of the type 
$\mathit{Whi^{(\prime)}}_{2n}$ consist of 
the $n$ 3-punctured spheres of the type $\widehat{\mathit{Whi^{(\prime)}}}_{n}$ 
and $n$ more 3-punctured spheres (shown in blue in the left of Figure~\ref{fig:u3ps-whibor}) 
which intersect the former ones. 
The type $\mathit{Whi^{(\prime)}}_{2n}$ can be regarded as a cyclic version of $B_{2n}$. 

The 3-punctured spheres of the types $\Whi_{2n}$ and $\Whip_{4n}$ 
are respectively contained only in the manifolds 
$\mathbb{W}_{n}$ and $\mathbb{W}^{\prime}_{2n}$. 
Suppose that a hyperbolic 3-manifold $M$ contains 3-punctured spheres of the type $\Whi_{2n}$. 
Since $M$ contains 3-punctured spheres of the type $\Whih_{n}$, 
it is obtained by a (possibly empty) Dehn filling on a cusp of $\mathbb{W}_{n}$. 
Then the Dehn filling must be empty, 
because $M$ has a cusp that does not intersect 
the 3-punctured spheres of the type $\Whih_{n}$. 
Thus $M$ is uniquely determined as $\mathbb{W}_{n}$. 
The same argument holds for $\Whip_{4n}$ in $\mathbb{W}^{\prime}_{2n}$.

\subsection*{$\mathit{\Bor}$}

The Borromean rings complement $\mathbb{W}^{\prime}_{2}$ 
contains six 3-punctured spheres of the type $Bor_{6}$ 
instead of $\Whip_{4}$. 
In order to display them, 
we put the Borromean rings 
so that each component is contained 
in a plane in $\mathbb{R}^{3}$. 
Then there are two 3-punctured spheres 
in the union of each plane with the infinite point 
as shown in the right of Figure~\ref{fig:u3ps-whibor}. 

The union of the 3-punctured spheres of the type $Bor_{6}$  with the adjacent torus cusps 
has a regular neighborhood whose boundary consists of spheres. 
The ambient hyperbolic 3-manifold is uniquely determined as $\mathbb{W}^{\prime}_{2}$ 
since it is irreducible.

\subsection*{$\Mag, \Tet, \Pen, \Oct$}

The 3-punctured spheres of the types $\Mag, \Tet, \Pen$, and $\Oct$ 
are respectively contained only in the manifolds 
$\mathbb{M}_{3}, \mathbb{M}_{4}, \mathbb{M}_{5}$, and $\mathbb{M}_{6}$. 
For $3 \leq n \leq 6$, the manifold $\mathbb{M}_{n}$ contains 
3-punctured spheres of the type $\widehat{\mathit{Whi^{(\prime)}}}_{n}$ 
and more 3-punctured spheres as shown in Figure~\ref{fig:u3ps-magtetpenoct}. 
Rotational symmetry gives the remaining 3-punctured spheres. 

Let $X$ be the union of the 3-punctured spheres of such a special type. 
The union of $X$ with the adjacent torus cusps 
has a regular neighborhood whose boundary consists of spheres. 
The ambient hyperbolic 3-manifold is uniquely determined 
since it is irreducible.

\subsection*{$\Teth, \Penh, \Octh$}

The 3-punctured spheres of the types $\Teth, \Penh$, and $\Octh$ 
are respectively contained only in the hyperbolic manifolds 
obtained by Dehn fillings on a cusp 
of $\mathbb{M}_{4}, \mathbb{M}_{5}$, and $\mathbb{M}_{6}$. 
In fact, such a surgered hyperbolic 3-manifold 
except $\mathbb{M}_{3}, \mathbb{M}_{4}$, and $\mathbb{M}_{5}$ 
contains no more 3-punctured spheres. 
The 3-punctured spheres of the types $\Teth, \Penh$, and $\Octh$ 
come from the ones of the types $\Tet, \Pen$, and $\Oct$ 
that are disjoint from the filled cusps. 

Let $X$ be the union of the 3-punctured spheres of such a special type. 
The union of $X$ with the adjacent torus cusps 
has a regular neighborhood homeomorphic to 
$\mathbb{M}_{4}, \mathbb{M}_{5} -$ two balls, or $\mathbb{M}_{6}$. 
The ambient hyperbolic 3-manifold is obtained by a Dehn filling on a cusp 
of $\mathbb{M}_{4}, \mathbb{M}_{5}$, or $\mathbb{M}_{6}$ 
since it is irreducible and atoroidal.

\subsection*{$B_{\infty}, \Whi_{\infty}$}

There are two types of the union of infinitely many 3-punctured spheres, 
which are infinite versions of $B_{2n}$. 
The 3-punctured spheres of the type $\Whi_{\infty}$ extend infinitely to both sides, 
and is contained only in the manifold $\mathbb{W}_{\infty}$ 
shown in the right of Figure~\ref{fig:u3ps-infinite}, 
which is an infinite cyclic cover of the Whitehead link complement $\mathbb{W}_{1}$. 
The 3-punctured spheres of the type $B_{\infty}$ extend infinitely to one side, 
and is contained in half of $\mathbb{W}_{\infty}$ 
shown in the left of Figure~\ref{fig:u3ps-infinite}. 

It is possible to consider infinite versions of $A_{n}$, 
but in such a case there is a cusp that bounds additional 3-punctured spheres. 
Then the union is $B_{\infty}$ or $\Whi_{\infty}$. 
We will prove it in Section~\ref{section:parameter}.

\subsection{Symmetries of $\Mag, \Tet, \Pen$, and $\Oct$}

The central 3-punctured sphere of $\mathbb{M}_{3}$ 
in Figure~\ref{fig:u3ps-magtetpenoct} is special 
in the sense that this is the unique 3-punctured sphere 
intersecting any other one at two geodesics. 
In contrast, 
for any pair of the 3-punctured spheres 
in $\mathbb{M}_{4}, \mathbb{M}_{5}$, or $\mathbb{M}_{6}$, 
the manifold has an isometry 
that maps one of the pair to the other. 

\fig[width=12cm]{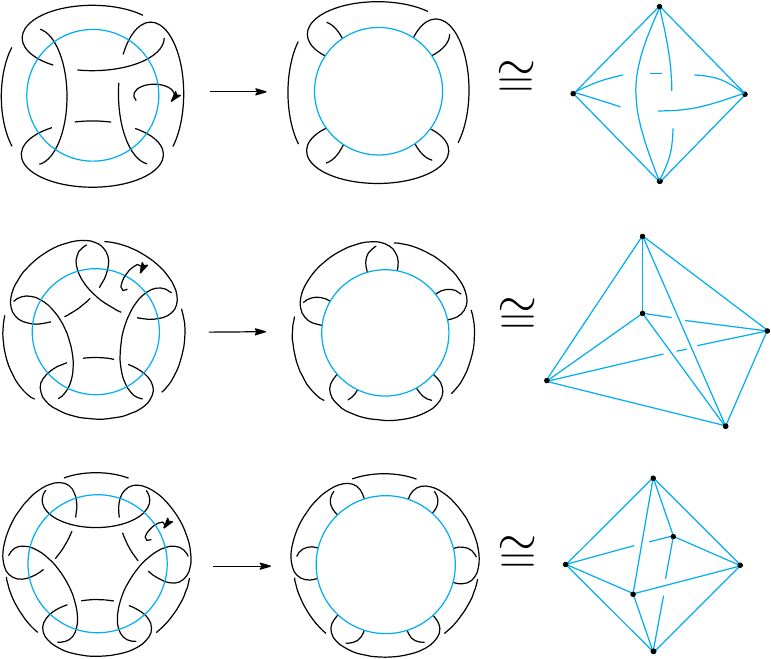}{Quotients of $\mathbb{M}_{4}, \mathbb{M}_{5}$, and $\mathbb{M}_{6}$}

Following Dunfield and Thurston~\cite{dunfield2003virtual}, 
we describe 
the manifolds $\mathbb{M}_{4}, \mathbb{M}_{5}$, and $\mathbb{M}_{6}$ 
with respect to their intrinsic symmetry. 
Each of $\mathbb{M}_{4}, \mathbb{M}_{5}$, and $\mathbb{M}_{6}$ 
has an involutional isometry 
that rotates about the blue circle in Figure~\ref{fig:u3ps-quotient}. 
The quotients by these involutions are 
naturally decomposed into ideal polyhedra. 
(In the deformations shown in the right side of Figure~\ref{fig:u3ps-quotient}, 
the black arcs shrink to the black vertices.) 
Then the original manifolds are recovered by the double branched coverings. 
The quotient of $\mathbb{M}_{5}$ is the boundary of a 4-dimensional simplex 
(a.k.a. a pentachoron) made of five regular ideal tetrahedra. 
The quotient of $\mathbb{M}_{6}$ 
is the double of a regular ideal octahedron. 
The quotient of $\mathbb{M}_{4}$ 
is decomposed into four ideal tetrahedra 
whose dihedral angles are $\pi/4, \pi/4$, and $\pi/2$. 
Each 3-punctured sphere 
in $\mathbb{M}_{4}, \mathbb{M}_{5}$, and $\mathbb{M}_{6}$ 
is the preimage of a face of these ideal polyhedra 
by the double branched covering. 
In particular, 
the manifolds $\mathbb{M}_{4}, \mathbb{M}_{5}$, and $\mathbb{M}_{6}$ 
have isometries that can map a 3-punctured sphere to any other one. 

If we cut the manifolds 
$\mathbb{M}_{3}, \mathbb{M}_{4}, \mathbb{M}_{5}$, and $\mathbb{M}_{6}$ 
along all the 3-punctured spheres, 
we respectively obtain 
two ideal triangular prisms, 
eight ideal tetrahedra 
each of which is a quarter of a regular ideal octahedron, 
ten regular ideal tetrahedra, 
and four regular ideal octahedra. 
By Sakuma and Weeks~\cite{sakuma1995examples}, 
these are the canonical decompositions 
in the sense of Epstein and Penner~\cite{epstein1988euclidean}.

\subsection{Graphs for the unions}

In Figure~\ref{fig:u3ps-graph}, 
we give graphs that indicate how the 3-punctured spheres intersect. 
The vertices of a graph correspond to the 3-punctured spheres. 
Two vertices are connected by an edge 
if the corresponding 3-punctured spheres intersect. 
Two vertices are connected by two edges 
if the corresponding 3-punctured spheres intersect at two geodesics. 
An edge is oriented 
if the corresponding intersection is separating in a 3-punctured sphere 
and non-separating in the other 3-punctured sphere. 
Our notations $T_{3}$ and $T_{4}$ come from the triangle and tripod of the graphs. 

\fig[width=12cm]{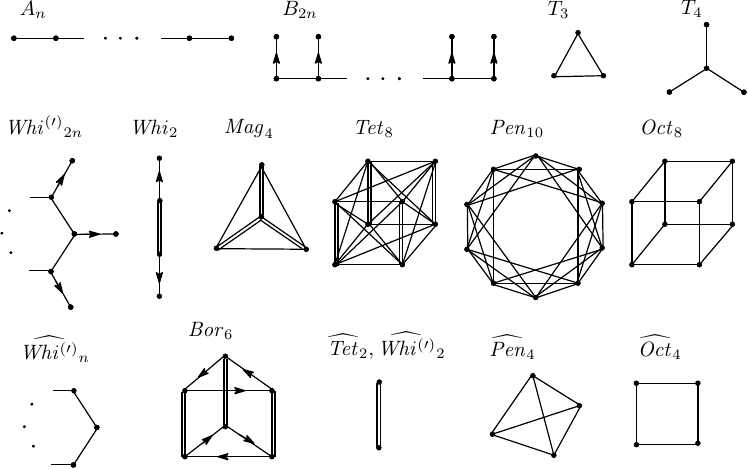}{Graphs indicating intersection of 3-punctured spheres}

\subsection{3-punctured spheres in a non-orientable hyperbolic 3-manifold}

We remark that the assumption of orientability is necessary. 
For instance, 
we can obtain a non-orientable hyperbolic 3-manifold $\mathbb{N}_{3}$
by gluing one regular ideal octahedron 
as shown in Figure~\ref{fig:u3ps-nonori}.  
The manifold $\mathbb{N}_{3}$ was given 
by Adams and Sherman~\cite{adams1991minimum} 
as the 3-cusped hyperbolic 3-manifold of minimal complexity. 
We remark that the manifold $\mathbb{M}_{5}$ is 
the 5-cusped hyperbolic 3-manifold of minimal complexity. 
The pairs of faces $A \cup B$ and $C \cup D$ 
are mapped to two 3-punctured spheres in $\mathbb{N}_{3}$. 
These 3-punctured spheres intersect at three geodesics. 
Such intersection does not appear 
in an orientable hyperbolic 3-manifold 
as we will show in Lemma~\ref{lem:3int}. 

\fig[width=5cm]{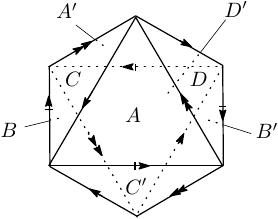}{Gluing of a regular ideal octahedron 
for $\mathbb{N}_{3}$}

\section{Proof of the classification}
\label{section:proof}

We begin to prove Theorem~\ref{thm:main}. 
Theorem~\ref{thm:infinite} will be proven in Section~\ref{section:parameter}. 
We consider totally geodesic embedded 3-punctured spheres 
in an orientable hyperbolic 3-manifold. 
For simplicity, 
we assume that every 3-punctured sphere is totally geodesic. 
We first consider the intersection of two 3-punctured spheres. 
After that,
we classify the union of 3-punctured spheres.

\subsection{Intersection of two 3-punctured spheres}
\label{subsection:intersection}

In this subsection, we classify the intersection of two 3-punctured spheres. 
The intersection of two 3-punctured spheres 
consists of disjoint simple geodesics. 
There are six simple geodesics in a 3-punctured sphere. 
Three of them are non-separating, 
and the other three are separating 
as shown in Figure~\ref{fig:u3ps-geodesic}. 
A component of the intersection of two 3-punctured spheres  
is a type (N,N), (S,N), or (S,S) 
depending on whether 
the geodesic is non-separating or separating in the two 3-punctured spheres. 
Of course, ``N'' and ``S'' respectively indicate non-separating and separating. 
In Lemma~\ref{lem:ss}, we will show that an (S,S)-intersection does not occur. 
The unions of the types $A_{2}$ and $B_{2}$, shown in Theorem~\ref{thm:main}, 
respectively contain an (N,N)-intersection and an (S,N)-intersection. 

\begin{prop}
\label{prop:intersection}
The intersection of two 3-punctured spheres 
in an orientable hyperbolic 3-manifold 
is one of the following types: 
\begin{itemize}
\item[(o)]
The intersection is empty, i.e. 
the two 3-punctured spheres are disjoint. 
\item[(i)]
The intersection consists of an (N,N)-intersection. 
\item[(ii)]
The intersection consists of an (S,N)-intersection. 
\item[(iii)]
The intersection consists of two (N,N)-intersections. 
\end{itemize}
\end{prop}

\fig[width=8cm]{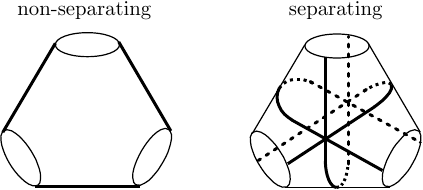}{The simple geodesics in a 3-punctured sphere}

We show that the other types of intersection are impossible. 
There are at most three disjoint simple geodesics 
in a 3-punctured sphere. 
We first show that the intersection of two 3-punctured spheres 
does not consist of three geodesics. 

\begin{lem}
\label{lem:3int}
Let $M$ be a (possibly non-orientable) hyperbolic 3-manifold. 
Suppose that $M$ contains 
3-punctured spheres $\Sigma_{0}$ and $\Sigma_{1}$ 
that intersect at three geodesics. 
Then $M$ is decomposed into a regular ideal octahedron 
along $\Sigma_{0}$ and $\Sigma_{1}$. 
Furthermore, $M$ is non-orientable. 
\end{lem}
\begin{proof}
We regard the hyperbolic 3-space $\mathbb{H}^{3}$ 
as the universal covering of $M$, 
and use the upper-half space model of $\mathbb{H}^{3}$. 
We denote by $\overline{(a,b)}$ the geodesic in $\mathbb{H}^{3}$ 
whose endpoints are 
$a, b \in \widehat{\mathbb{C}} = \partial \mathbb{H}^{3}$. 
Let $\widetilde{\Sigma}_{i} \subset \mathbb{H}^{3}$ denote 
the preimage of $\Sigma_{i}$ for $i=0,1$. 
Let $\widetilde{\Sigma}_{0}^{0}$ be 
a component of $\widetilde{\Sigma}_{0}$. 
The plane $\widetilde{\Sigma}_{0}^{0}$ contains an ideal triangle $\Delta$ 
whose edges are lifts of the intersection $\Sigma_{0} \cap \Sigma_{1}$. 
We may assume that the vertices of $\Delta$ are 
$0, 1, \infty \in \widehat{\mathbb{C}} = \partial \mathbb{H}^{3}$. 
Hence there are components $\widetilde{\Sigma}_{1}^{k}$ 
of $\widetilde{\Sigma}_{1}$ for $k=0,1,2$ 
such that 
$\widetilde{\Sigma}_{0}^{0} \cap \widetilde{\Sigma}_{1}^{0} 
= \overline{(0,\infty)}, 
\widetilde{\Sigma}_{0}^{0} \cap \widetilde{\Sigma}_{1}^{1} 
= \overline{(1,\infty)}$, and 
$\widetilde{\Sigma}_{0}^{0} \cap \widetilde{\Sigma}_{1}^{2} 
= \overline{(0,1)}$. 
Since $\Sigma_{1}$ is embedded in $M$, 
the three planes 
$\widetilde{\Sigma}_{1}^{0}, \widetilde{\Sigma}_{1}^{1}$, and 
$\widetilde{\Sigma}_{1}^{2}$ are mutually disjoint. 
Therefore $\widetilde{\Sigma}_{1}^{0}, \widetilde{\Sigma}_{1}^{1}$, and 
$\widetilde{\Sigma}_{1}^{2}$ orthogonally 
intersect $\widetilde{\Sigma}_{0}^{0}$. 

In the same manner, 
there are components 
$\widetilde{\Sigma}_{0}^{1}$ and $\widetilde{\Sigma}_{0}^{2}$ 
of $\widetilde{\Sigma}_{0}$ 
such that 
$\widetilde{\Sigma}_{0}^{1} \cap \widetilde{\Sigma}_{1}^{0} 
= \overline{(ai,\infty)}$ 
and $\widetilde{\Sigma}_{0}^{2} \cap \widetilde{\Sigma}_{1}^{0} 
= \overline{(0,ai)}$ 
for $a>0$. 
The planes $\widetilde{\Sigma}_{0}^{1}$ and 
$\widetilde{\Sigma}_{0}^{2}$ orthogonally 
intersect $\widetilde{\Sigma}_{1}^{0}$. 

We continue the same argument. 
Since each of the preimages of $\Sigma_{0}$ and $\Sigma_{1}$ 
consists of disjoint planes, 
we have $a=1$. 
There are components $\widetilde{\Sigma}_{j}^{3}$ 
of $\widetilde{\Sigma}_{j}$ 
such that $\widetilde{\Sigma}_{0}^{3} \cap \widetilde{\Sigma}_{1}^{1} 
= \overline{(1,1+i)}$ 
and $\widetilde{\Sigma}_{0}^{1} \cap \widetilde{\Sigma}_{1}^{3} 
= \overline{(i,1+i)}$. 
Figure~\ref{fig:u3ps-lift} shows 
the boundary of these planes in $\widehat{\mathbb{C}}$. 
The planes $\widetilde{\Sigma}_{j}^{k}$ for $j=0,1$ and $k=0,1,2,3$ 
bound a regular ideal octahedron. 
The other components of 
$\widetilde{\Sigma}_{0}$ and $\widetilde{\Sigma}_{1}$ 
are disjoint from this regular ideal octahedron. 
An isometry of a regular ideal octahedron has a fixed point 
in the interior. 
Consequently, 
if $M$ is decomposed along $\Sigma_{0}$ and $\Sigma_{1}$, 
one of the components after the decomposition is 
a regular ideal octahedron. 
Now the surface area of a regular ideal octahedron is equal to $8\pi$, 
and the area of a 3-punctured sphere is equal to $2\pi$. 
Therefore there are no other components after the decomposition. 

\fig[width=6cm]{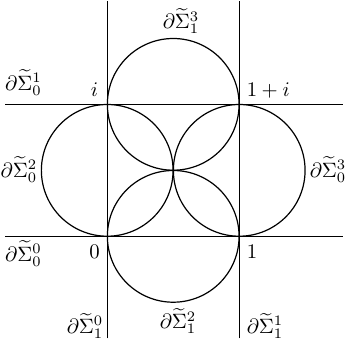}{Boundary of components of $\widetilde{\Sigma}_{0}$ and $\widetilde{\Sigma}_{1}$}

According to the cusped octahedral census 
by Goerner~\cite[Table 2]{goerner2017census}, 
there are two orientable hyperbolic 3-manifolds 
and six non-orientable hyperbolic 3-manifolds 
obtained from one regular ideal octahedron. 
Agol~\cite{agol2010minimal} showed that 
the Whitehead link complement and the $(-2,3,8)$-pretzel link complement 
have the smallest volume $V_{oct}$ 
of the orientable hyperbolic 3-manifold with two cusps, 
and described the ways 
to glue a regular ideal octahedron 
to obtain the two manifolds. 
Assume that $M$ is orientable. 
Then $M$ is one of the above two manifolds. 
The 3-punctured spheres $\Sigma_{0}$ and $\Sigma_{1}$ 
are the images of the faces of a regular ideal octahedron. 
The gluing ways, however, do not give two 3-punctured spheres 
from the faces of a regular ideal octahedron. 
\end{proof}

Therefore 
the intersection of two 3-punctured spheres 
in an orientable hyperbolic 3-manifold
cannot consist of three geodesics. 
We orient the hyperbolic 3-manifold and the 3-punctured spheres. 
We can easily show the following lemma 
by considering the orientation of a neighborhood of the arc. 

\begin{lem}
\label{lem:parity}
Let $S$ and $T$ be properly embedded oriented surfaces 
in an oriented 3-manifold with boundary. 
Suppose that $S$ and $T$ intersect transversally. 
Let an arc $g$ be a component of $S \cap T$. 
Let $x_{0}$ and $x_{1}$ denote the endpoints of $g$. 
For $i=0$ and $1$, 
let $s_{i}$ denote the boundary components of $S$ containing $x_{i}$, 
which possibly coincide. 
In the same manner, 
let $t_{i}$ 
denote the boundary components of $T$ containing $x_{i}$. 
Then 
the intersections of $s_{i}$ and $t_{i}$ at $x_{i}$ 
have opposite signs with respect to the induced orientation. 
\end{lem}

We return the cases of two 3-punctured spheres 
in an oriented hyperbolic 3-manifold. 
Note that the boundary components of a 3-punctured sphere 
are closed geodesics in a cusp with respect to its Euclidean metric. 

\begin{lem}
\label{lem:ss}
The intersection of two 3-punctured spheres 
contain no (S,S)-intersection. 

\end{lem}
\begin{proof}
Assume that a geodesic $g$ in the intersection is 
separating in both 3-punctured spheres. 
Note that there might be another component of the intersection. 
We consider a cusp containing the endpoints of $g$. 
Let $s$ and $t$ denote 
the boundary components of the two 3-punctured spheres 
that intersect $g$. 
Then the intersection of the loops $s$ and $t$ 
contains at least the endpoints of $g$. 
Since the loops $s$ and $t$ are contained in a common cusp, 
the signs of the intersection at these endpoints coincide. 
This is impossible by Lemma~\ref{lem:parity}. 
\end{proof}

We suppose that the intersection of two 3-punctured spheres 
consists of two geodesics. 

\begin{lem}
\label{lem:nn-ns}
The intersection of two 3-punctured spheres does not consist of 
one (N,N)-intersection and one (S,N)-intersection. 
\end{lem}
\begin{proof}
Assume that two 3-punctured spheres intersect 
at one (N,N)-intersection and one (S,N)-intersection. 
There are two possibilities of the intersection 
as shown in the left and center of Figure~\ref{fig:u3ps-beta}. 

In the left case, 
a cusp contains three loops $s,u$, and $v$ of boundary components 
of the 3-punctured spheres. 
The loops $u$ and $v$ are disjoint. 
The loops $s$ and $u$ intersect at two points, 
but the loops $s$ and $v$ intersect at one point. 
This is impossible. 

\fig[width=12cm]{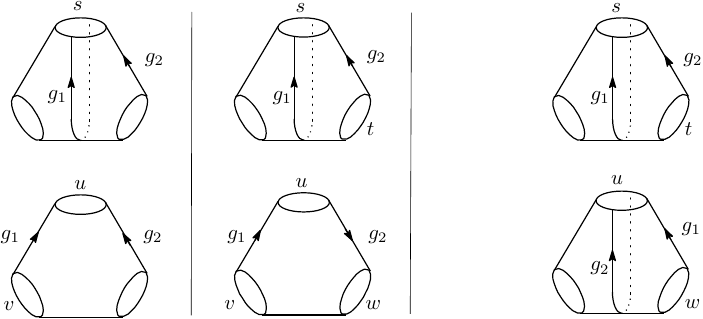}{Orientations of two geodesics containing an (S,N)-intersection}

In the central case, 
a cusp contains five loops $s,t,u,v$, and $w$ of boundary components 
of the 3-punctured spheres. 
The two loops $s$ and $t$ are disjoint, 
and the three loops $u,v$, and $w$ are mutually disjoint, 
Hence the number of intersection of these five loops is a multiple of six. 
This contradicts the fact that the intersection consists of 
the four endpoints of the two geodesics $g_{1}$ and $g_{2}$. 
\end{proof}

\begin{lem}
\label{lem:ns-ns}
The intersection of two 3-punctured spheres 
does not consist of two (S,N)-intersections. 
\end{lem}
\begin{proof}
Assume that two 3-punctured spheres intersect 
at two (S,N)-intersections. 
It is sufficient to consider the intersection 
as shown in the right of Figure~\ref{fig:u3ps-beta}. 

In the case, 
a cusp contains four loops $s,t,u$, and $w$ of the boundary components 
of the 3-punctured spheres. 
Then $s \cap t = \emptyset, u \cap w = \emptyset, t \cap w = \emptyset$, 
whereas $s \cap u \neq \emptyset$. 
This is impossible.
\end{proof}

\begin{proof}[Proof of Proposition~\ref{prop:intersection}]
We have excluded the cases other than 
the cases in Proposition~\ref{prop:intersection}. 
\end{proof}

\subsection{The unions with the type (iii)-intersection}
\label{subsection:type3}

In this subsection, we classify the unions of 3-punctured spheres 
that contain the type (iii)-intersection 
shown in Proposition~\ref{prop:intersection}. 

\begin{lem}
\label{lem:type3}
If two 3-punctured spheres have the type (iii)-intersection, 
the ambient hyperbolic 3-manifold is obtained 
by a (possibly empty) Dehn filling 
from one of the following manifolds (Figure~\ref{fig:u3ps-type3}): 
\begin{itemize}
\item the manifold $\mathbb{W}_{2}$, 
which is a double covering of the Whitehead link complement, 

\item the Borromean rings complement $\mathbb{W}^{\prime}_{2}$, or 

\item the minimally twisted hyperbolic 4-chain link complement 
$\mathbb{M}_{4}$. 
\end{itemize}
\end{lem}

\fig[width=12cm]{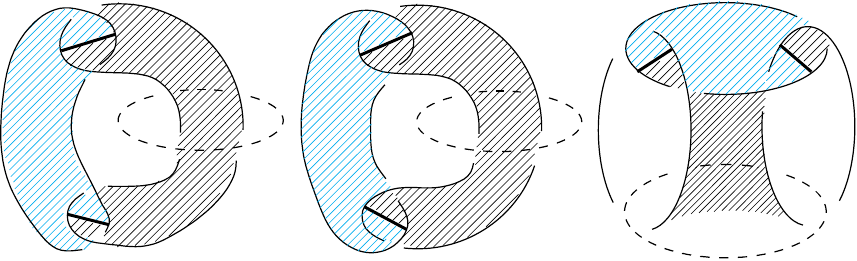}{Two 3-punctured spheres intersecting at two geodesics}

Note that each of the manifolds 
$\mathbb{W}_{2}, \mathbb{W}^{\prime}_{2}$, and $\mathbb{M}_{4}$ 
is obtained by gluing two regular ideal octahedra, 
and hence they have a common volume. 

\begin{proof}
There are three possibilities depending on the orientations 
of intersectional geodesics as shown in Figure~\ref{fig:u3ps-2alpha}. 

\fig[width=11cm]{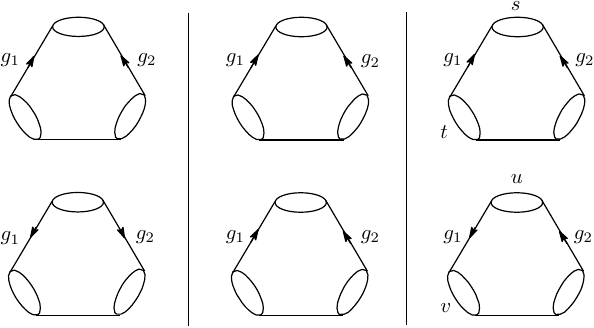}{Orientations of two (N,N)-intersections}

The left case gives the union of two 3-punctured spheres 
of the types $\Whih_{2}$ and $\Whiph_{2}$ 
depending on the signs of the intersections on the boundary. 
The union of the two 3-punctured spheres with the adjacent cusps 
has a regular neighbourhood whose boundary is a torus. 
Hence the ambient 3-manifold is obtained 
by a (possibly empty) Dehn filling 
from $\mathbb{W}_{2}$ or $\mathbb{W}^{\prime}_{2}$.

\fig[width=4cm]{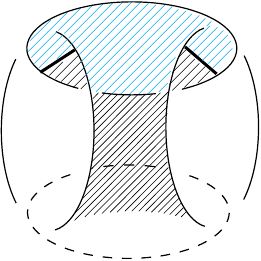}{The non-hyperbolic minimally twisted 4-chain link}

The central case also gives two types of the union 
depending on the signs of the intersections on the boundary. 
One of these is the type $\Teth$. 
Then the ambient 3-manifold is obtained 
by a (possibly empty) Dehn filling from $\mathbb{M}_{4}$ 
in the same manner as above. 
The other type occurs in a manifold 
obtained by a (possibly empty) Dehn filling 
from the non-hyperbolic minimally twisted 4-chain link complement 
$\mathbb{M}^{\prime}_{4}$ 
as shown in Figure~\ref{fig:u3ps-nonhyp}. 
The manifold $\mathbb{M}^{\prime}_{4}$ 
can be decomposed along a torus 
into two copies of $\Sigma \times S^{1}$, 
where $\Sigma$ is a 3-punctured sphere. 
Hence the ambient 3-manifold is a graph manifold, 
which is not hyperbolic. 

The right case does not occur. 
Assume that it occurs. 
Then a cusp contains four boundary components $s,t,u$, and $v$ 
of the 3-punctured spheres. 
We have $s \cap t = \emptyset$ and $u \cap v = \emptyset$, 
but the intersection in $s,t,u$, and $v$ 
consists of three points. 
This is impossible. 
\end{proof}

In Lemmas~\ref{lem:bor6}--\ref{lem:tet2}, 
we will show that 
if two 3-punctured spheres have the type (iii)-intersection 
shown in Proposition~\ref{prop:intersection}, 
then the union of all the 3-punctured spheres is 
$\Whi_{4}, \Whih_{2}, \Bor, \Whiph_{2}, \Tet, \Teth$, or $\Mag$. 
We need Corollary~\ref{cor:disjoint} and Lemma~\ref{lem:homologous} 
to prove that there are no other 3-punctured spheres than the described ones 
in certain manifolds. 

\begin{thm}[Miyamoto~\cite{miyamoto1994volumes}]
\label{thm:estimate}
Let $N$ be a hyperbolic 3-manifold with totally geodesic boundary. 
Then 
\[
\mathrm{vol} (N) \geq \frac{1}{2} |\chi (\partial N)| V_{oct}, 
\]
where $\chi$ indicates the Euler characteristic, and 
$V_{oct} = 3.6638...$ is the volume of a regular ideal octahedron. 
\end{thm}

\begin{cor}
\label{cor:disjoint}
A hyperbolic 3-manifold $M$ contains at most $\lfloor \mathrm{vol}(M) / V_{oct} \rfloor$ 
disjoint 3-punctured spheres. 
\end{cor}
\begin{proof}
By cutting $n$ disjoint totally geodesic 3-manifolds, 
we obtain a hyperbolic 3-manifold $N$ with totally geodesic boundary 
such that $\chi (\partial N) = -2n$. 
Then we apply Theorem~\ref{thm:estimate} to $N$. 
\end{proof}

\begin{lem}
\label{lem:homologous}
Let $M$ be an orientable hyperbolic 3-manifold. 
Suppose that two distinct totally geodesic 3-punctured spheres 
$\Sigma$ and $\Sigma^{\prime}$ in $M$ 
represent a common homology class in $H_{2}(M, \partial M; \mathbb{Z})$. 
Then $\Sigma$ and $\Sigma^{\prime}$ are disjoint. 
\end{lem}
\begin{proof}
It is sufficient to show that 
$\partial \Sigma \cap \partial \Sigma^{\prime} = \emptyset$. 
For $i=0,1,2$, let $s_{i}$ and $s^{\prime}_{i}$ 
denote the components of $\partial \Sigma$ and $\partial \Sigma^{\prime}$. 
The unions of loops $\partial \Sigma = s_{0} \cup s_{1} \cup s_{2}$ and 
$\partial \Sigma^{\prime} 
= s^{\prime}_{0} \cup s^{\prime}_{1} \cup s^{\prime}_{2}$ 
represent a common homology class in $H_{1}(\partial M; \mathbb{Z})$. 

We first suppose the homology classes of any pair of $s_{0}, s_{1}$, and $s_{2}$ 
do not cancel in $H_{1}(\partial M; \mathbb{Z})$. 
By changing indices, we may assume that 
$[s_{i}] = [s^{\prime}_{i}] \in H_{1}(\partial M; \mathbb{Z})$. 
Since the loops $s_{i}$ and $s^{\prime}_{j}$ do not coincide, 
we have $\partial \Sigma \cap \partial \Sigma^{\prime} = \emptyset$. 

Otherwise we may assume that 
$[s_{0}] = -[s_{1}]$ and $[s^{\prime}_{0}] = -[s^{\prime}_{1}]$ 
in $H_{1}(\partial M; \mathbb{Z})$. 
If $s_{0} \cup s_{1}$ and $s^{\prime}_{0} \cup s^{\prime}_{1}$ intersect, 
then $\Sigma$ and $\Sigma^{\prime}$ have the type (iii)-intersection, 
and the loops $s_{0}, s_{1}, s^{\prime}_{0}$, and $s^{\prime}_{1}$ 
are contained in a common cusp. 
This is impossible by Lemma~\ref{lem:type3}. 
Therefore we have $\partial \Sigma \cap \partial \Sigma^{\prime} = \emptyset$. 
\end{proof}

In order to apply Corollary~\ref{cor:disjoint} and Lemma~\ref{lem:homologous}, 
we use the Thurston norm. 
For a surface $S = \bigsqcup_{i} S_{i}$ 
(each $S_{i}$ is a connected component), 
we define $\chi_{-}(S) = \sum_{i} \max \{0, -\chi (S_{i})\}$. 
For an orientable compact 3-manifold $M$, 
the Thurston norm of a class $\sigma \in H_{2}(M, \partial M; \mathbb{Z})$ 
is defined to be $\| \sigma \| = \min \chi_{-}(S)$, 
where the minimum is taken  
over the (possibly disconnected) embedded surfaces $S$ 
that represent the class $\sigma$. 
Thurston~\cite{thurston1986norm} showed that 
$\| \cdot \|$ is extended to a norm on $H_{2}(M, \partial M; \mathbb{R})$ 
for a hyperbolic 3-manifold, 
and the unit norm ball 
$\{\sigma \in H_{2}(M, \partial M; \mathbb{R}) | \| \sigma \| \leq 1\}$ 
is convex. 
Note that 
the norm $\| \sigma \|$ of an integer class $\sigma$ is odd 
if and only if 
$\sigma$ can be represented by an essential surface $S$ 
such that the number of the components of $\partial S$ is odd. 

\begin{lem}
\label{lem:bor6}
The Borromean rings complement $\mathbb{W}^{\prime}_{2}$ 
has exactly the six 3-punctured spheres of the type $\Bor$. 
\end{lem}
\begin{proof}
Thurston~\cite{thurston1986norm} described 
the unit Thurston norm ball for $\mathbb{W}^{\prime}_{2}$ 
as follows. 
We may assume that linearly independent classes 
$x,y,z \in H_{2}(\mathbb{W}^{\prime}_{2}, 
\partial \mathbb{W}^{\prime}_{2}; \mathbb{R})$ 
are represented by three of the 3-punctured spheres 
of the type $\Bor$ 
as described in Section~\ref{section:description}. 
The classes $x,y$ and $z$ form a basis of 
$H_{2}(\mathbb{W}^{\prime}_{2}, 
\partial \mathbb{W}^{\prime}_{2}; \mathbb{R}) \cong \mathbb{R}^{3}$. 
Moreover, $\| x \| = \| y \| = \| z \| = 1$. 
The eight classes $\pm x \pm y \pm z$ are fibered class, 
i.e. they are represented by a fiber of a mapping torus. 
Hence each of the classes $\pm x \pm y \pm z$ 
is not represented by a 3-punctured sphere. 
Moreover, 
each of the classes $\pm x \pm y \pm z$ 
is represented by a surface with at least three boundary components.  
Hence $\| \pm x \pm y \pm z \| \neq 1$. 
Therefore we have $\| \pm x \pm y \pm z \| = 3$. 
The 14 classes $\pm x, \pm y, \pm z$, and $(\pm x \pm y \pm z)/3$ 
are contained in the boundary of the unit Thurston norm ball. 
This fact and the convexity imply that 
the unit Thurston norm ball is the octahedron 
whose vertices are $\pm x, \pm y$, and $\pm z$. 

We know that $\pm x, \pm y$, and $\pm z$ are exactly the integer classes 
in the boundary of the unit Thurston norm ball. 
Hence any 3-punctured sphere in $\mathbb{W}^{\prime}_{2}$ 
represents $\pm x, \pm y$, or $\pm z$. 
By ignoring the orientation, 
we may state that 
for each of $x,y$, and $z$, 
there are two 3-punctured spheres 
each of which represents the class. 
Therefore it is sufficient to show that 
$x,y$ and $z$ cannot be represented by any other 3-punctured sphere. 
Assume that $\mathbb{W}^{\prime}_{2}$ has 
another 3-punctured sphere $\Sigma$. 
Then Lemma~\ref{lem:homologous} implies that 
$\Sigma$ is disjoint from the two 3-punctured spheres 
representing the same class. 
This contradicts Corollary~\ref{cor:disjoint} and the fact that 
$\mathrm{vol} (\mathbb{W}^{\prime}_{2}) = 2V_{oct}$. 
Therefore $\mathbb{W}^{\prime}_{2}$ has no other 3-punctured spheres 
than the described ones of $\Bor$. 
\end{proof}

\begin{lem}
\label{lem:whi4}
The manifold $\mathbb{W}_{2}$ 
has exactly the four 3-punctured spheres of the type $\Whi_{4}$. 
\end{lem}
\begin{proof}
Since $\mathrm{vol}(\mathbb{W}_{2}) = 2V_{oct}$, 
the manifold $\mathbb{W}_{2}$ contains 
at most two disjoint 3-punctured spheres 
by Corollary~\ref{cor:disjoint}. 
Let $\Sigma_{1}, \dots , \Sigma_{4}$ be 
the 3-punctured spheres of $\Whi_{4}$, 
and let $C_{1}, \dots , C_{3}$ be the cusps of $\mathbb{W}_{2}$ 
as shown in Figure~\ref{fig:u3ps-whi4}. 
In order to show that there are no other 3-punctured spheres, 
we describe the unit Thurston norm ball for $\mathbb{W}_{2}$. 
Let $x \in H_{2}(\mathbb{W}_{2}, \partial \mathbb{W}; \mathbb{R})$ 
denote the class represented by each of $\Sigma_{1}$ and $\Sigma_{2}$, 
and let $y,z \in H_{2}(\mathbb{W}_{2}, \partial \mathbb{W}; \mathbb{R})$ 
denote the classes respectively represented 
by $\Sigma_{3}$ and $\Sigma_{4}$. 
The classes $x,y$, and $z$ form a basis of 
$H_{2}(\mathbb{W}_{2}, \partial \mathbb{W}; \mathbb{R}) 
\cong \mathbb{R}^{3}$. 
Moreover, $\| x \| = \| y \| = \| z \| = 1$. 
Consider the class $x+y+z$. 
A surface representing $x+y+z$ intersects the three cusps. 
If we show that $x+y+z$ cannot be represented by a 3-punctured sphere, 
we have $\| x+y+z \| =3$. 

Assume that $x+y+z$ is represented by a 3-punctured sphere $\Sigma$. 
The boundary slopes of $\Sigma$ are determined, 
and each cusp contains one of these slopes. 
The intersection $\Sigma \cap \Sigma_{1}$ 
is disjoint from the cusp $C_{1}$. 
Since each cusp contains exactly one boundary component of $\Sigma$, 
the intersection $\Sigma \cap \Sigma_{i}$ for $i=1,3$ 
is not a separating geodesic in $\Sigma_{i}$. 
Hence $\Sigma$ and $\Sigma_{1}$ are disjoint. 
Similarly, $\Sigma$ and $\Sigma_{2}$ are also disjoint. 
This contradicts the fact that $\mathbb{W}_{2}$ contains 
at most two disjoint 3-punctured spheres. 

In the same manner, 
we have $\| \pm x \pm y \pm z \| = 3$. 
Hence the unit Thurston norm ball for $\mathbb{W}_{2}$ 
is the octahedron whose vertices are $\pm x, \pm y$, and $\pm z$ 
similarly to $\mathbb{W}^{\prime}_{2}$. 
Therefore any 3-punctured sphere in $\mathbb{W}_{2}$ 
represents $\pm x, \pm y$, or $\pm z$. 

\fig[width=8cm]{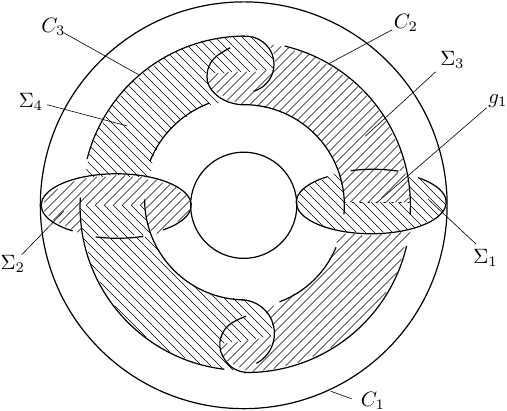}{The 3-punctured spheres of the type $Whi_{4}$}

Now the class $x$ is represented only by $\Sigma_{1}$ and $\Sigma_{2}$. 
Assume that 
$\mathbb{W}_{2}$ has another 3-punctured sphere $\Sigma$ 
representing $y$ than $\Sigma_{3}$. 
Then $\Sigma \cap \Sigma_{3} = \emptyset$. 
Since the homology classes represented 
by the components of $\partial \Sigma_{3}$ 
do not cancel, 
the boundary $\partial \Sigma$ consists of 
a loop in $C_{2}$ and two loops in $C_{3}$. 
Hence $\Sigma$ intersects $\Sigma_{1}$ 
at the geodesic $g_{1} = \Sigma_{1} \cap \Sigma_{3}$. 
This contradicts the fact that $\Sigma \cap \Sigma_{3} = \emptyset$. 
Hence there are no other 3-punctured spheres representing $y$ 
than $\Sigma_{3}$. 
The same argument holds for $z$. 
Therefore $\mathbb{W}_{2}$ has no other 3-punctured spheres
than $\Sigma_{1}, \dots , \Sigma_{4}$. 
\end{proof}

\begin{lem}
\label{lem:tet8}
The manifold $\mathbb{M}_{4}$ 
has exactly the eight 3-punctured spheres of the type $\Tet$. 
\end{lem}
\begin{proof}
We first describe the unit Thurston norm ball for $\mathbb{M}_{4}$. 
Let $\Sigma_{1}, \dots , \Sigma_{4}$ be the 3-punctured spheres 
as shown in Figure~\ref{fig:u3ps-tet8}, 
which respectively represent 
$x,y,z,w \in 
H_{2}(\mathbb{M}_{4}, \partial \mathbb{M}_{4}; \mathbb{R})$. 
Here the orientations of these 3-punctured spheres 
are induced by the projection to the diagram. 
Let $\Sigma_{5}, \dots , \Sigma_{8}$ denote the 3-punctured spheres 
that respectively represent $y+z+w, -x+z+w, x+y-w$, and $x+y+z$. 
These eight 3-punctured spheres are the ones described 
in Section~\ref{section:description}. 

Let $u_{1}=(x+y)/2, u_{2}=(y+z)/2, u_{3}=(z+w)/2$, and $u_{4}=(x-w)/2$. 
The classes $u_{1}, \dots , u_{4}$ form a basis of 
$H_{2}(\mathbb{M}_{4}, \partial \mathbb{M}_{4}; \mathbb{R}) 
\cong \mathbb{R}^{4}$. 
With respect to this basis, 
we can present classes as 
\begin{align*}
x=(1,-1,1,1), \ y=(1,1,-1,-1), &\ z=(-1,1,1,1), \ w=(1,-1,1,-1), \\
y+z+w=(1,1,1,-1), &\ -x+z+w=(-1,1,1,-1), \\
x+y-w=(1,1,-1,1), &\ x+y+z=(1,1,1,1). 
\end{align*}
Since the norms of $u_{1}, \dots , u_{4}, x,y,z,w, y+z+w,-x+z+w,x+y-w$, 
and $x+y+z$ 
are equal to one, 
the convexity implies that 
the unit Thurston norm ball for $\mathbb{M}_{4}$ 
is the 4-dimensional cube whose vertices are 
$\pm x, \pm y, \pm z, \pm w, 
\pm (y+z+w), \pm (-x+z+w), \pm (x+y-w)$, and $\pm (x+y+z)$. 
Therefore any 3-punctured sphere in $\mathbb{M}_{4}$ 
represents one of these classes. 

\fig[width=8cm]{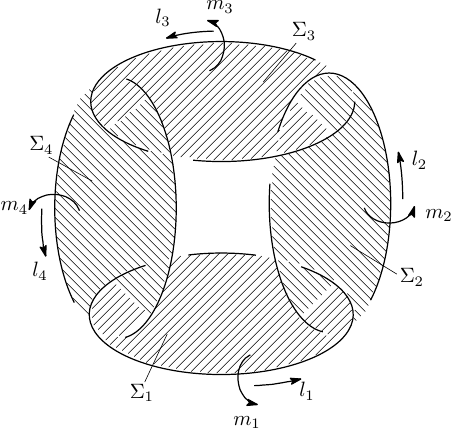}{Four 3-punctured spheres in $\mathbb{M}_{4}$}

As we showed in Section \ref{section:description}, 
for any pair of $\Sigma_{1}, \dots , \Sigma_{8}$, 
there is an isometry of $\mathbb{M}_{4}$ 
that maps one of the pair to the other. 
Hence it is sufficient to show that $\Sigma_{1}$ is 
the unique 3-punctured sphere representing $x$. 

Since $\mathrm{vol}(\mathbb{M}_{4}) = 2V_{oct}$, 
the manifold $\mathbb{M}_{4}$ contains 
at most two disjoint 3-punctured spheres 
by Corollary~\ref{cor:disjoint}. 
Assume that $\mathbb{M}_{4}$ has another 3-punctured sphere $\Sigma$ 
representing $x$ than $\Sigma_{1}$. 
Lemma~\ref{lem:homologous} implies that 
$\Sigma \cap \Sigma_{1} = \emptyset$. 
Since the components of $\partial \Sigma$ are contained in distinct cusps, 
their three slopes are the same as the slopes of $\partial \Sigma_{1}$. 
Hence $\Sigma \cap \Sigma_{3} = \emptyset$. 
This contradicts the fact that 
$\mathbb{M}_{4}$ contains at most two disjoint 3-punctured spheres. 
\end{proof}

\begin{lem}
\label{lem:chain2}
Let $M$ be a hyperbolic 3-manifold obtained by a (non-empty) Dehn filling 
on the cusp of $\mathbb{W}_{2}$ or $\mathbb{W}^{\prime}_{2}$ 
as in Lemma~\ref{lem:type3}. 
Then $M$ has exactly the two 3-punctured spheres 
of the type $\Whih_{2}$ or $\Whiph_{2}$ respectively. 
\end{lem}
\begin{proof}
The Mayer-Vietoris sequence and the Poincar\'{e} duality imply that 
\\ $H_{2}(M, \partial M; \mathbb{R}) \cong \mathbb{R}^{2}$. 
The manifold $M$ contains at least two 3-punctured spheres 
of the type $\Whih_{2}$ or $\Whiph_{2}$. 
They represent classes $x$ and $y$ 
that form a basis of $H_{2}(M, \partial M; \mathbb{R})$. 
Since $\| x \| = \| y \| =1$ and $\| x+y \| = \| x-y \| =2$, 
the unit Thurston norm ball for $M$ is the square 
whose vertices are $\pm x$ and $\pm y$. 
Hence any 3-punctured sphere in $M$ represents $\pm x$ or $\pm y$. 
Since a hyperbolic Dehn surgery decreases 
the volume~\cite[Theorem 6.5.6]{thurston1978geometry}, 
we have $\vol (M) < 2V_{oct}$. 
Hence the manifold $M$ does not contain two disjoint 3-punctured spheres 
by Corollary~\ref{cor:disjoint}. 
Therefore $M$ has no other 3-punctured spheres. 
\end{proof}

\begin{lem}
\label{lem:tet2}
Let $M$ be a hyperbolic 3-manifold obtained by a (non-empty) Dehn filling 
on a cusp of $\mathbb{M}_{4}$. 
If $M = \mathbb{M}_{3}$, 
it has exactly the four 3-punctured spheres of the type $\Mag$. 
Otherwise 
$M$ has exactly the two 3-punctured spheres of the type $\Teth$. 
\end{lem}
\begin{proof}
Thurston~\cite{thurston1986norm} described 
the unit Thurston norm ball for $\mathbb{M}_{3}$ 
as follows. 
We may assume that the four known 3-punctured spheres represent 
$x,y,z,x+y+z \in 
H_{2}(\mathbb{M}_{3}, \partial \mathbb{M}_{3}; \mathbb{R})$. 
Then the unit Thurston norm ball is the parallelepiped 
whose vertices are $\pm x, \pm y, \pm z$, and $\pm (x+y+z)$. 
Since $\vol (\mathbb{M}_{3}) < 2V_{oct}$, 
the manifold $\mathbb{M}_{3}$ has no other 3-punctured spheres 
by Corollary~\ref{cor:disjoint}. 

We orient the meridians $m_{i}$ and the longitudes $l_{i}$ 
of the cusps of $\mathbb{M}_{4}$ as shown in Figure~\ref{fig:u3ps-tet8}. 
For coprime integers $p \geq 0$ and $q$, 
let $M$ be a hyperbolic 3-manifold obtained by the $(p,q)$-Dehn filling 
on the first cusp of $\mathbb{M}_{4}$. 
In other words, $M$ is obtained by gluing a solid torus 
to $\mathbb{M}_{4}$ along the slope $pm_{1}+ql_{1}$ as meridian. 
If $(p,q)=(0,1), (1,0), (1,1)$, or $(2,1)$, then $M$ is not hyperbolic. 
If $(p,q)=(1,-1), (1,2), (3,1)$, or $(3,2)$, then $M=\mathbb{M}_{3}$. 
We exclude these cases. 
Note that in general four Dehn fillings give a common 3-manifold 
by the symmetry of $\mathbb{M}_{4}$. 

Following the notation in Lemma~\ref{lem:tet8}, 
the two 3-punctured spheres $\Sigma_{3}$ and $\Sigma_{5}$ 
in $\mathbb{M}_{4}$ are disjoint from the filled cusp. 
Their union is $\Teth$. 
Assume that $M$ contains another 3-punctured sphere $\Sigma$ 
than $\Sigma_{3}$ or $\Sigma_{5}$. 
We may assume that 
$\Sigma$ is the union 
of an $n$-punctured sphere $\Sigma^{\prime}$ in $\mathbb{M}_{4}$ 
and $(n-3)$ essential disks in the filled solid torus. 
Suppose that $\Sigma^{\prime}$ represents 
$ax+by+cz+dw \in H_{2}(\mathbb{M}_{4}, \partial \mathbb{M}_{4}; \mathbb{R})$ 
for $a,b,c,d \in \mathbb{Z}$. 
Note that $n$ is odd if and only if $a+b+c+d$ is odd. 
Then the homology classes of the boundaries are 
\begin{align*}
[\partial \Sigma^{\prime}]
&= ((b-d)m_{1}+al_{1})+((a+c)m_{2}+bl_{2}) \\
&+((b+d)m_{3}+cl_{3})+((-a+c)m_{4}+dl_{4}), \\
[\partial \Sigma_{3}] &= m_{2}+l_{3}+m_{4}, \\
[\partial \Sigma_{5}] &= (m_{2}+l_{2})+(2m_{3}+l_{3})+(m_{4}+l_{4}). 
\end{align*}
For $i=3$ and $5$, 
two components of $\partial \Sigma$ and $\partial \Sigma_{i}$ 
intersect in at most two points 
by Proposition~\ref{prop:intersection}. 
Hence 
\begin{align*}
\biggl| \det 
\begin{pmatrix}
a+c & 1 \\ 
b & 0
\end{pmatrix}
\biggr|
= |b| \leq 2, \
&\biggl| \det 
\begin{pmatrix}
a+c & 1 \\ 
b & 1
\end{pmatrix}
\biggr|
= |a-b+c| \leq 2, \\
\biggl| \det 
\begin{pmatrix}
b+d & 0 \\ 
c & 1
\end{pmatrix}
\biggr|
= |b+d| \leq 2, \
&\biggl| \det 
\begin{pmatrix}
b+d & 2 \\ 
c & 1
\end{pmatrix}
\biggr|
= |b-2c+d| \leq 2, \\
\biggl| \det 
\begin{pmatrix}
-a+c & 1 \\ 
d & 0
\end{pmatrix}
\biggr|
= |d| \leq 2, \
&\biggl| \det 
\begin{pmatrix}
-a+c & 1 \\ 
d & 1
\end{pmatrix}
\biggr|
= |a-c+d| \leq 2. 
\end{align*}

Suppose that $(b-d,a)=(0,0)$. 
Then $|b| \leq 1, |b-c| \leq 1$, and $n$ is odd. 
Since $c$ is odd, we have $(b,c)=\pm (0,1)$ or $\pm (1,1)$. 
If $(a,b,c,d)=\pm (0,0,1,0)$, 
then $\Sigma$ is disjoint from $\Sigma_{3}$. 
If $(a,b,c,d)=\pm (0,1,1,1)$, 
then $\Sigma$ is disjoint from $\Sigma_{5}$. 
Since $\vol (M) < 2V_{oct}$, these are impossible by Corollary~\ref{cor:disjoint}. 

Suppose that $(b-d,a) \neq (0,0)$. 
Then $p/q = (b-d)/a$. 
By reversing the orientation if necessary, 
we may assume that $b-d \geq 0$. 
Then $(b-d)/p$ is odd if and only if $n$ is even. 
Therefore it is sufficient to consider the cases 
\begin{align*}
(a,b,c,d)=
&(-1,0,-1,-2), (-1,2,1,0), (1,2,\pm 1,-2), \\
&(3,0,-1,-2), (3,2,1,0), (3,2,\pm 1,-2). 
\end{align*}
Suppose that $(a,b,c,d)=(-1,0,-1,-2)$. 
Then $(a+c,b)=(-2,0), (b+d,c)=(-2,-1)$, and $(-a+c,d)=(0,-2)$. 
Hence $\partial \Sigma^{\prime}$ has at least five components 
disjoint from the filled cusp. 
This is impossible. 
The other cases are also impossible similarly. 
\end{proof}

\subsection{The unions with the type (ii)-intersection}
\label{subsection:type2}

In this subsection, we classify the unions of 3-punctured spheres
that contain the type (ii)-intersection shown in Proposition~\ref{prop:intersection}. 

\begin{lem}
\label{lem:type2}
Suppose that two 3-punctured spheres $\Sigma_{1}$ and $\Sigma_{2}$ 
in a hyperbolic 3-manifold $M$ 
have the type (ii)-intersection. 
In other words, we may assume that 
$\Sigma_{1} \cap \Sigma_{2}$ is separating in $\Sigma_{1}$ 
and non-separating in $\Sigma_{2}$. 
Let $X$ be the component of the union of the 3-punctured spheres in $M$ 
that contains $\Sigma_{1}$ and $\Sigma_{2}$. 
Then $X$ is $B_{2n}, \Whi_{2n}, \Whip_{4n}$, or $\Bor$. 
(We allow $2n = \infty$.) 
\end{lem}
\begin{proof}
We consider the two 3-punctured spheres $\Sigma_{1}$ and $\Sigma_{2}$ 
as shown in Figure~\ref{fig:u3ps-b2}. 
Suppose that another 3-punctured sphere $\Sigma$ intersects $\Sigma_{2}$. 
Consider the intersection of $\Sigma$ and the cusp $C$. 
Proposition~\ref{prop:intersection} implies that 
$\Sigma \cap \Sigma_{2} \cap C$ consists of at most two points. 
If $\Sigma \cap \Sigma_{2} \cap C = \emptyset$, 
then $\Sigma \cap \Sigma_{2} = g_{2}, \Sigma \cap \Sigma_{1} = g_{3}$, 
and $\Sigma \cap C = \emptyset$. 
Otherwise $\Sigma \cap \Sigma_{2} = g_{4} \cup g_{5}$ 
and $\Sigma \cap \Sigma_{1} = \emptyset$. 
In both cases, the ambient 3-manifold $M$ is 
the Borromean rings complement $\mathbb{W}^{\prime}_{2}$. 
Then $X$ is $\Bor$ by Lemma~\ref{lem:bor6}. 

Suppose that $\Sigma_{2}$ is disjoint from 
any other 3-punctured sphere than $\Sigma_{1}$. 
If $\Sigma_{1}$ is also disjoint from any other 3-punctured sphere 
than $\Sigma_{2}$, 
then $X$ is $B_{2}$. 
If $\Sigma_{1}$ intersects another 3-punctured sphere $\Sigma_{3}$, 
there is another 3-punctured sphere $\Sigma_{4}$ 
that is homologous to $\Sigma_{2}$ and intersects only $\Sigma_{3}$. 
We can continue this argument. 
If 3-punctured spheres lies cyclically, 
then $X$ is $\Whi_{2n}$ or $\Whip_{4n}$. 
Otherwise $X$ is $B_{2n}$. 
If $X$ consists of infinitely many 3-punctured spheres, 
then $X$ is $B_{\infty}$ or $\Whi_{\infty}$. 
\end{proof}

\fig[width=6cm]{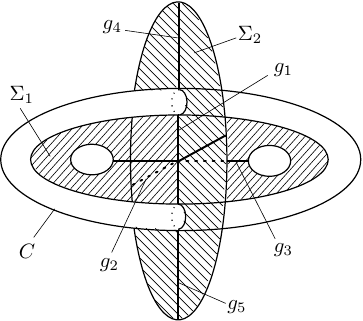}{3-punctured spheres of the type $B_{2}$}

\subsection{The unions without intersection of the type (ii) or (iii)}
\label{subsection:type1}

From now on, 
we consider a component $X$ of the union of 3-punctured spheres 
without intersection of the type (ii) or (iii). 
Let us consider the intersection $L = X \cap C$, 
where $C$ is a cusp. 
For the intersection of a 3-punctured sphere and a cusp, 
we call a component of it a \textit{boundary loop}, 
which is a closed geodesic in the cusp with respect to the Euclidean metric. 
Then $L$ is the union of the boundary loops in the cusp $C$. 

\begin{lem}
\label{lem:boundary}
Let $X$ be a component of the union of 3-punctured spheres 
without intersection of the type (ii) or (iii). 
Then the intersection of two boundary loops consists of at most one point. 
Moreover, each boundary loop intersects other boundary loops at most two points. 
\end{lem}
\begin{proof}
The first assertion follows from the assumption that 
any pair of 3-punctured spheres in $X$ has the type (i)-intersection. 
The second assertion follows from the fact that 
each boundary component of a 3-punctured sphere meets 
exactly two non-separating simple geodesics. 
\end{proof}

By Lemma~\ref{lem:boundary}, 
we may assume that the slope of a boundary loop in $L$ is $0,1$, or $\infty$ 
with respect to a choice of meridian and longitude. 
We say that the following types of $L$ are general (Figure~\ref{fig:u3ps-genslope}): 
\begin{itemize}
\item disjoint simple boundary loops, 
\item two boundary loops with one common point, and 
\item three boundary loops with two common points, two of which are parallel. 
\end{itemize}
If $L$ is not general, 
then $L$ contains boundary loops 
that are one of the three special types shown in Figure~\ref{fig:u3ps-spslope}. 
Figures~\ref{fig:u3ps-genslope} and \ref{fig:u3ps-spslope} 
show fundamental domains of the cusp $C$. 

\fig[width=8cm]{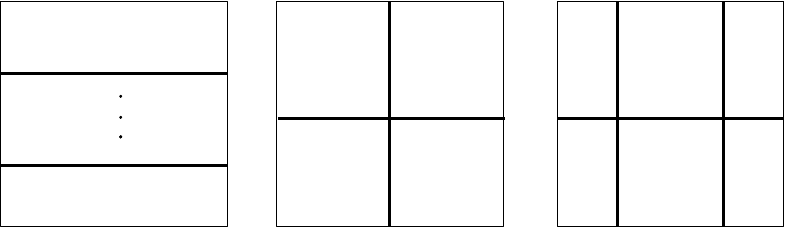}{General types of $L$}

\fig[width=8cm]{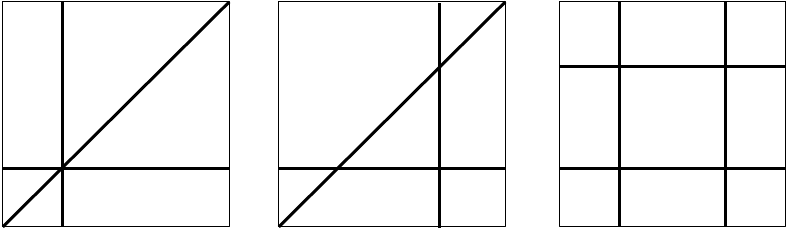}{Special types of $L$}

We first consider special types of $L$.

\begin{lem}
\label{lem:spunion}
Let $X$ and $L$ be as above. 
\begin{itemize}
\item
If $L$ contains three loops of slopes $0,1$, and $\infty$ 
with common intersection, 
then $X$ contains $T_{3}$. 

\item
If $L$ contains three loops of slopes $0,1$, and $\infty$ 
without common intersection, 
then $X$ contains $\Penh$. 

\item
If $L$ contains two pairs of loops of slopes 0 and $\infty$, 
then $X$ contains $\Octh$. 
\end{itemize}

\end{lem}
\begin{proof}
We obtain the asserted containments 
by manually combining 3-punctured spheres. 
In the second case, 
the union of the three 3-punctured spheres with $C$ 
has a regular neighborhood that contains another 3-puncture. 
Moreover, the union of these four 3-punctured spheres is $\Penh$. 
In the last case, 
the assertion follows from Lemma~\ref{lem:oct4}. 
\end{proof}

\begin{lem}
\label{lem:oct4}
Suppose that $L$ contains two pairs of loops of slopes 0 and $\infty$. 
Let $\Sigma_{1}, \dots , \Sigma_{4}$ be 3-punctured spheres in $X$ 
such that $\Sigma_{1} \cap C$ and $\Sigma_{3} \cap C$ are of slope 0, 
and $\Sigma_{2} \cap C$ and $\Sigma_{4} \cap C$ are of slope $\infty$. 
Then $\Sigma_{1} \cap \Sigma_{3} = \emptyset$ 
and $\Sigma_{2} \cap \Sigma_{4} = \emptyset$. 
\end{lem}
\begin{proof}
Assume that two 3-punctured spheres $\Sigma_{1}$ and $\Sigma_{3}$ 
intersect at a geodesic $g_{5}$. 
There are two cases depending on orientations of geodesics 
at the intersection as shown in Figure~\ref{fig:u3ps-oct4}. 

\fig[width=12cm]{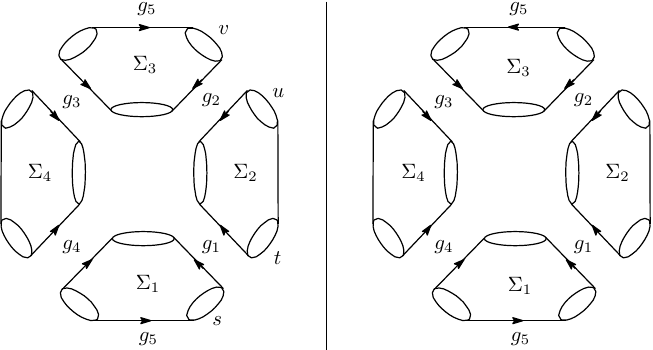}{Orientations of geodesics in four 3-punctured spheres}

In the left case, 
four loops $s,t,u$, and $v$ are contained in a common cusp. 
Then $s \cap u = \emptyset$ and $t \cap u = \emptyset$, 
whereas $s \cap t \neq \emptyset$. 
This is impossible. 

In the right case, 
the union of $\Sigma_{1}, \Sigma_{2}$, and $\Sigma_{3}$ is $\Whih_{3}$. 
Hence the ambient 3-manifold $M$ is obtained by a Dehn filling 
on a cusp of $\mathbb{W}_{3}$. 
Then the proof is completed by Lemma~\ref{lem:whi3}. 
\end{proof}

\begin{lem}
\label{lem:whi3}
Let $M$ be a hyperbolic 3-manifold 
obtained by a (non-empty) Dehn filling 
on the cusp $C_{1}$ of $\mathbb{W}_{3}$. 
Suppose that $M \neq \mathbb{M}_{3}$. 
Then $M$ has exactly the three 3-punctured spheres 
of the type $\Whih_{3}$. 
\end{lem}
\begin{proof}
We orient the meridians $m_{i}$ and the longitudes $l_{i}$ 
of the cusps of $\mathbb{W}_{3}$ 
as shown in Figure~\ref{fig:u3ps-whi3}. 
The 3-punctured spheres 
$\Sigma_{1}, \Sigma_{2}, \Sigma_{3}$, and $\Sigma_{4}$ 
respectively represent classes $x,y,z$, and $w$ that form a basis of 
$H_{2}(\mathbb{W}_{3}, \partial \mathbb{W}_{3}; \mathbb{R}) 
\cong \mathbb{R}^{4}$. 
Assume that $M$ has another 3-punctured sphere $\Sigma$ 
than $\Sigma_{2}, \Sigma_{3}$, or $\Sigma_{4}$. 
Then we may assume that $\Sigma$ is the union of 
an $n$-punctured sphere $\Sigma^{\prime}$ in $\mathbb{W}_{3}$ 
and $(n-3)$ essential disks in the filled solid torus. 
Suppose that $\Sigma^{\prime}$ represents 
$ax+by+cz+dw \in H_{2}(\mathbb{W}_{3}, \partial \mathbb{W}_{3}; \mathbb{R})$ 
for $a,b,c,d \in \mathbb{Z}$. 
Then 
\[ [\partial \Sigma^{\prime}] 
= al_{1}+((-c-d)m_{2}+bl_{2})+((-b-d)m_{3}+cl_{3})+((-b-c)m_{4}+dl_{4}). 
\]
If $a \neq 0$, then $M$ is obtained by the $(0,1)$-Dehn filling, 
which is not hyperbolic. 
Hence we have $a=0$, and $b+c+d$ is odd. 
Since we exclude $\mathbb{M}_{3}$, 
two components of $\partial \Sigma$ and $\partial \Sigma_{i}$ 
for $i=2,3,4$ intersect in at most one point. 
Hence we have 
\[| b | \leq 1, \ | c | \leq 1, \ | d | \leq 1, \
| c+d | \leq 1, \ | b+d | \leq 1, \ | b+c | \leq 1.
\] 
These inequalities imply that 
$(b,c,d)=(\pm 1,0,0), (0,\pm 1,0)$, or $(0,0, \pm 1)$. 

Therefore we may assume that $\Sigma$ is 
homologous to $\Sigma_{2}$ in $M$. 
Lemma~\ref{lem:homologous} implies that 
$\Sigma \cap \Sigma_{2} = \emptyset$, 
but we have $\Sigma \cap \Sigma_{3} = \Sigma_{2} \cap \Sigma_{3}$ 
and $\Sigma \cap \Sigma_{4} = \Sigma_{2} \cap \Sigma_{4}$. 
This is impossible. 
\end{proof}

\fig[width=8cm]{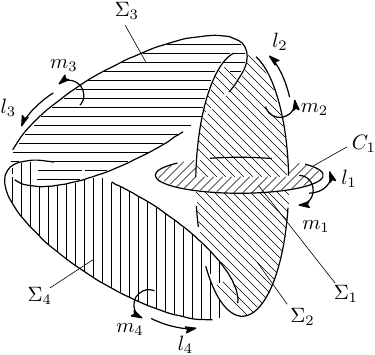}{Four 3-punctured spheres in $\mathbb{W}_{3}$}

Based on Lemma~\ref{lem:spunion}, 
we classify the unions $X$ that contain special types of $L$. 
We recall that two 3-punctured spheres of $X$ intersect in 
at most one geodesic. 

\begin{lem}
\label{lem:specialcusp}
Let $X$ be a component of the union of the 3-punctured spheres 
without intersection of the type (ii) or (iii) 
in Proposition~\ref{prop:intersection}. 
\begin{itemize} 
\item If $X$ contains $T_{3}$, 
then $X$ is $T_{3}$ or $\Pen$. 
\item If $X$ contains $\Penh$, 
then $X$ is $\Penh$ or $\Pen$. 
\item If $X$ contains $\widehat{Oct}_{4}$, 
then $X$ is $\Octh, \Oct$, or $\Pen$. 
\end{itemize}
\end{lem}
\begin{proof}
We first show that the manifold $\mathbb{M}_{5}$ 
has exactly the ten 3-punctured spheres of $\Pen$. 
The intersection of the 3-punctured spheres of $\Pen$ 
and each cusp consists of six boundary loops 
of slopes $0,1$, and $\infty$. 
If $\mathbb{M}_{5}$ has another 3-punctured sphere, 
a cusp contains a more boundary loop. 
This contradicts Lemma~\ref{lem:boundary}. 
Therefore $\mathbb{M}_{5}$ has no other 3-punctured spheres. 

Suppose that $X$ contains $T_{3}$ 
and another 3-punctured sphere. 
Then the intersection of $X$ and a cusp 
contains three loops of slopes $0,1$ and $\infty$ 
without common intersection. 
Hence $X$ contains $\Penh$. 
The ambient 3-manifold is obtained by a possibly empty Dehn filling 
on a cusp of $\mathbb{M}_{5}$. 
Since the ambient 3-manifold contains 3-punctured spheres of $T_{3}$, 
its volume is at least $\vol (\mathbb{T}_{3}) = \vol (\mathbb{M}_{5})$. 
Therefore the ambient 3-manifold is $\mathbb{M}_{5}$, 
and $X$ is $\Pen$. 

Suppose that $X$ contains $\Penh$ and another 3-punctured sphere. 
Since $X$ also contains $T_{3}$, the union $X$ is $\Pen$. 

Suppose that $X$ contains $\Octh$. 
The ambient 3-manifold is obtained by a possibly empty Dehn filling 
on a cusp of $\mathbb{M}_{6}$. 
The manifold $\mathbb{M}_{6}$ has 
exactly the eight 3-punctured spheres of the type $Oct_{8}$, 
for otherwise $X$ contains $T_{3}$ and $\Penh$. 

\fig[width=8cm]{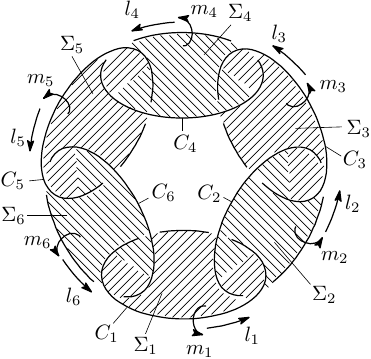}{Six 3-punctured spheres in $\mathbb{M}_{6}$}

Let $M$ be a hyperbolic 3-manifold obtained by a Dehn filling 
on the cusp $C_{1}$ of $\mathbb{M}_{6}$. 
Assume that $M$ is not $\mathbb{M}_{5}$, 
and has another 3-punctured sphere $\Sigma$ 
than the ones of $\Octh$. 
We orient the meridians $m_{i}$ and the longitudes $l_{i}$ 
of the cusps of $\mathbb{M}_{6}$ 
as shown in Figure~\ref{fig:u3ps-oct8}. 
The 3-punctured spheres $\Sigma_{1}, \dots , \Sigma_{6}$ 
respectively represent 
classes $x_{1}, \dots , x_{6}$ that form a basis of 
$H_{2}(\mathbb{M}_{6}, \partial \mathbb{M}_{6}; \mathbb{R}) 
\cong \mathbb{R}^{6}$. 
We may assume that $\Sigma$ is the union 
of an $n$-punctured sphere $\Sigma^{\prime}$ in $\mathbb{M}_{6}$ 
and $(n-3)$ essential disks in the filled solid torus. 
Suppose that $\Sigma^{\prime}$ represents 
$a_{1}x_{1 }+ \dots + a_{6}x_{6} \in 
H_{2}(\mathbb{M}_{6}, \partial \mathbb{M}_{6}; \mathbb{R})$ 
for $a_{1}, \dots , a_{6} \in \mathbb{Z}$. 
Then 
\begin{align*}
[\partial \Sigma^{\prime}] 
&= ((-a_{6}+a_{2})m_{1}+a_{1}l_{1}) 
+((a_{1}-a_{3})m_{2}+a_{2}l_{2}) 
+((-a_{2}+a_{4})m_{3}+a_{3}l_{3}) \\
&+((a_{3}-a_{5})m_{4}+a_{4}l_{4}) 
+((-a_{4}+a_{6})m_{5}+a_{5}l_{5})
+((a_{5}-a_{1})m_{6}+a_{6}l_{6}).
\end{align*} 
The 3-punctured sphere $\Sigma$ is disjoint from the cusp $C_{4}$, 
for otherwise $X$ is $\Pen$. 
By the same reason, 
each $\Sigma \cap C_{i}$ for $i=2,3,5,6$ 
is a multiple of $m_{i}$ or $l_{i}$. 
Since $\Sigma$ is a 3-punctured sphere, 
we have  
\[
| a_{1}-a_{3} | + | a_{2} | 
+ | -a_{2}+a_{4} | + | a_{3} | 
+ | -a_{4}+a_{6} | + | a_{5} |
+ | a_{5}-a_{1} | + | a_{6} | = 1 \, \text{or} \, 3.
\] 
However, 
\[
(a_{1}-a_{3}) + a_{2} 
+ (-a_{2}+a_{4}) + a_{3} 
+ (-a_{4}+a_{6}) + a_{5} 
+ (a_{5}-a_{1}) + a_{6} = 2(a_{5}+a_{6})
\] 
is even. 
This is impossible. 
\end{proof}

Thus we have classified the unions that contain special boundary loops. 
In the remaining cases, 
the intersection $L$ is general (Figure~\ref{fig:u3ps-genslope}). 
Note that if $L$ consists three boundary loops with two common points, 
the two parallel loops in $L$ are contained in distinct 3-punctured spheres.

\begin{lem}
\label{lem:trivalent}
Suppose that the intersection of $X$ and any cusp 
is general. 
If $X$ contains 3-punctured sphere $\Sigma$ 
that intersects three other 3-punctured spheres, 
then $X$ is $T_{4}$. 
\end{lem}
\begin{proof}
Clearly $X$ contains $T_{4}$. 
Let $C_{1}, C_{2}$, and $C_{3}$ denote the cusps 
meeting the 3-punctured sphere $\Sigma$. 
Assume that there is another 3-punctured sphere $\Sigma$ 
than the ones of $T_{4}$. 
Then $\Sigma$ intersects at least one of the three cusps 
$C_{1}, C_{2}$, and $C_{3}$. 
This contradicts the assumption 
that the intersection of $X$ and a cusp is general. 
\end{proof}

In the last remaining cases, 
the 3-punctured spheres in $X$ are placed linearly or cyclically. 
Suppose that $X$ consists of finitely many 3-punctured spheres. 
If the 3-punctured spheres in $X$ are placed linearly, 
then $X$ is $A_{n}$. 
If the 3-punctured spheres in $X$ are placed cyclically, 
then $X$ is $\Whih_{n}$ or $\Whiph_{2n}$. 

Thus we complete the proof of Theorem~\ref{thm:main}.

\section{Volume and number of 3-punctured spheres}
\label{section:volume}
As an application of Theorem~\ref{thm:main}, 
we estimate the number of 3-punctured spheres 
in a hyperbolic 3-manifold by its volume. 
We recall that if a hyperbolic 3-manifold $M$ contains 
$n$ disjoint 3-punctured spheres, then $\vol (M) \geq n V_{oct}$ 
by Corollary~\ref{cor:disjoint}.

\begin{thm}
\label{thm:3psvsvol}
Suppose that an orientable hyperbolic 3-manifold $M$ 
has $k$ 3-punctured spheres. 
Then $k \leq 4 \vol (M)/ V_{oct}$. 
The equality holds if and only if 
$M$ is the manifold $\mathbb{M}_{4}$. 
\end{thm}
\begin{proof}

We first consider the special cases. 
Let $V_{oct} = 3.6638...$ be the volume of a regular ideal octahedron, 
and let $V_{3} = 1.0149...$ be the volume of a regular ideal tetrahedron. 
The assertion for the special manifolds 
is obtained from the following inequalities: 
\begin{itemize}
\item $\vol (\mathbb{W}_{n}) = \vol (\mathbb{W}^{\prime}_{n}) 
= nV_{oct} > \frac{n}{2}V_{oct}$ 
for $\Whi_{2n}$ and $\Whip_{4n}$, 
\item $\vol (\mathbb{W}^{\prime}_{2}) = 2V_{oct} > \frac{3}{2}V_{oct}$ 
for $\Bor$, 
\item $\vol (\mathbb{M}_{3}) = 5.3334... > V_{oct}$ for $\Mag$, 
\item $\vol (\mathbb{M}_{4}) = 2V_{oct}$ for $\Tet$, 
\item $\vol (\mathbb{M}_{5}) = 10V_{3} > \frac{5}{2}V_{oct}$ for $\Pen$, 
\item $\vol (\mathbb{M}_{6}) = 4V_{oct} > 2V_{oct}$ for $\Oct$. 
\end{itemize}

Since a hyperbolic 3-manifold with 3-punctured spheres 
of the type $\Whih_{n}$ 
contains $\lfloor n/2 \rfloor$ disjoint 3-punctured spheres, 
its volume is at least $\lfloor n/2 \rfloor V_{oct}$. 
The same argument holds for $\Whiph_{2n}$. 
For $\Teth$, 
the volume of a hyperbolic 3-manifold obtained by a Dehn filling 
on a cusp of $\mathbb{M}_{4}$ 
is at least $V_{oct}$, 
since the manifold contains a 3-punctured sphere. 
For $\Penh$ and $\Octh$, 
the volume of a hyperbolic 3-manifold obtained by a Dehn filling 
on a cusp of $\mathbb{M}_{5}$ or $\mathbb{M}_{6}$ 
is at least $2V_{oct}$. 
Indeed, such a manifold has at least 4 cusps,  
and $\mathbb{M}_{4}$ has the smallest volume 
of the orientable hyperbolic 3-manifolds with at least 4 cusps 
\cite{yoshida2013minimal}. 
Thus we have shown the assertion for the special cases. 

We consider the general cases. 
Suppose that the union of the 3-punctured spheres 
of an orientable hyperbolic 3-manifold $M$ consists of 
the types $A_{n}, B_{2n}, \\ T_{3}$, and $T_{4}$. 
Let $a_{n}, b_{2n}, t_{3}$, and $t_{4}$ denote 
the number of the corresponding components. 
At least one of $a_{n}, b_{2n}, t_{3}$, and $t_{4}$ is positive. 
Then $M$ contains 
$\sum_{n} (\lfloor (n+1)/2 \rfloor a_{n} + nb_{2n}) + t_{3} + 3t_{4}$ 
disjoint 3-punctured spheres. 
Hence $\vol (M) \geq 
(\sum_{n} (\lfloor (n+1)/2 \rfloor a_{n} + nb_{2n}) + t_{3} + 3t_{4})V_{oct}$. 
The assertion follows from the inequality 
\[
\sum_{n} (\lfloor \frac{n+1}{2} \rfloor a_{n} + nb_{2n}) + t_{3} + 3t_{4}
> \frac{1}{4} \Bigl( \sum_{n} (n a_{n} + 2nb_{2n}) + 3t_{3} + 4t_{4} \Bigr), 
\]
which is easily checked 
by comparing the coefficients termwise. 
\end{proof}

\section{Bound of modulus for $A_{n}$}
\label{section:parameter}

Neighborhoods of 3-punctured spheres 
of the types $B_{2n}, T_{3}$, and $T_{4}$ 
are isometrically determined as the manifolds 
$\mathbb{B}_{n+1}, \mathbb{T}_{3}$, and $\mathbb{T}_{4}$. 
The metric of a neighborhood of 3-punctured spheres of the type $A_{n}$ 
for $n \geq 2$, however,  
depends on the ambient hyperbolic 3-manifold. 
Let us consider 3-punctured spheres $\Sigma_{1}, \dots , \Sigma_{n}$ 
of the type $A_{n}$. 
There are $n$ cusps each of which intersects 
two or three of $\Sigma_{1}, \dots , \Sigma_{n}$ 
at loops of two slopes. 
We will call them 
the \textit{adjacent torus cusps} for $\Sigma_{1}, \dots , \Sigma_{n}$. 
We define the meridians and longitudes of the adjacent torus cusps 
as the intersection of the cusps and the 3-punctured spheres, 
so that each 3-punctured sphere meets exactly one longitude. 
Then the meridians and the longitudes are uniquely determined 
if $n \geq 3$. 
For $A_{2}$, however, 
there is ambiguity to permute the meridian and longitude. 
In this case we take them arbitrarily. 

The Euclidean structure of such an adjacent cusp determines 
its modulus $\tau$ with respect to the meridian and longitude. 
Then the cusp is isometric to the quotient of $\mathbb{C}$ 
under the additive action of 
$\{ m+n\tau \in \mathbb{C} | m,n \in  \mathbb{Z} \}$, 
where 1 and $\tau$ respectively correspond to the meridian and longitude. 
We may assume that $\mathrm{Im} (\tau ) >0$ 
by taking an appropriate orientation. 
We first show that the moduli of such adjacent cusps coincide. 

\begin{prop}
\label{prop:modulus}
Suppose that an orientable hyperbolic 3-manifold $M$ contains 
two 3-punctured spheres $\Sigma_{1}$ and $\Sigma_{2}$ of the type $A_{2}$. 
Let $\tau$ and $\tau^{\prime}$ denote 
the moduli of the two adjacent cusps $C_{1}$ and $C_{2}$. 
Then $\tau = \tau^{\prime}$. 
\end{prop}
\begin{proof}
Let $x,y,z,w \in \pi_{1}(M)$ be represented by the loops 
shown in Figure~\ref{fig:u3ps-modulus}. 
The base point is taken in $\Sigma_{1} \cap \Sigma_{2}$. 
The meridians correspond to $y$ and $z$, 
and the longitudes correspond to $x$ and $w$. 
Regard $\pi_{1}(M)$ as a subgroup of 
$\mathrm{PSL}(2,\mathbb{C}) \cong \mathrm{Isom}^{+}(\mathbb{H}^{3})$. 
Since $x$ and $y$ are parabolic elements with distinct fixed points 
in $\partial \mathbb{H}^{3}$, 
we may assume that 
\[
x = \begin{pmatrix}
1 & 2 \\ 
0 & 1
\end{pmatrix}
\quad \text{and} \quad
y = \begin{pmatrix}
1 & 0 \\ 
c & 1
\end{pmatrix} 
\]
by taking conjugates. 
Then 
$xy^{-1} = \begin{pmatrix}
1-2c & 2 \\ 
-c & 1
\end{pmatrix}$ 
is also parabolic. 
Hence $|\mathrm{tr}(xy^{-1})| = |2-2c| =2$. 
Since $y$ is not the identity, we have $c=2$. 

\fig[width=8cm]{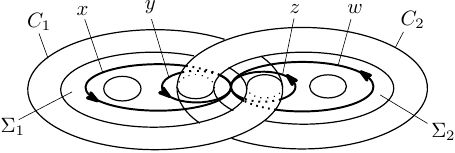}{Generators of $\pi_{1} (\Sigma_{1})$ and $\pi_{1} (\Sigma_{2})$}

The moduli $\tau$ and $\tau^{\prime}$ give the representations 
\[
z = \begin{pmatrix}
1 & 2/\tau \\ 
0 & 1
\end{pmatrix} 
\quad \text{and} \quad 
w = \begin{pmatrix}
1 & 0 \\ 
2\tau^{\prime} & 1
\end{pmatrix}. 
\]
Then 
$zw^{-1} = \begin{pmatrix}
1-4(\tau^{\prime}/\tau) & 2/\tau \\ 
-2\tau^{\prime} & 1 
\end{pmatrix}$ 
is also parabolic. 
Hence \\ $|\mathrm{tr}(zw^{-1})| = |2-4(\tau^{\prime}/\tau)| =2$. 
Since $\tau^{\prime} \neq 0$, we have $\tau = \tau^{\prime}$. 
\end{proof}

Therefore 
the metric of a neighbourhood of 3-puncture spheres 
of the type $A_{n} (n \geq 2)$ 
is determined by the single modulus $\tau$. 
In particular, the angle at the intersection is equal to $\arg \tau$. 

Let $\mathcal{C}_{n}$ denote the set 
of the moduli for 3-punctured spheres of the type $A_{n}$ 
contained in (possibly infinite volume) orientable hyperbolic 3-manifolds. 
We first give a bound for $\mathcal{C}_{n}$ 
by using the Shimizu-Leutbecher lemma. 

\begin{lem}[the Shimizu-Leutbecher lemma~\cite{shimizu1963discontinuous}]
\label{lem:shimizu}
Suppose that a group generated by two elements 
\[
\begin{pmatrix}
1 & 1 \\ 
0 & 1
\end{pmatrix}, 
\begin{pmatrix}
a & b \\ 
c & d
\end{pmatrix} 
\in \mathrm{PSL}(2,\mathbb{C})
\]
is discrete. 
Then $c=0$ or $|c| \geq 1$. 
\end{lem}

\begin{prop}
\label{prop:modulusbound}
Let $\tau \in \mathcal{C}_{2}$. 
Then 
\[
|m \tau +n| \geq \frac{1}{4}
\quad \text{and} \quad 
\left|\frac{m}{\tau} +n\right| \geq \frac{1}{4}
\]
for any $(m,n) \in \mathbb{Z} \times \mathbb{Z} \setminus \{(0,0)\}$.
In particular, 
\[
\frac{1}{4} \leq |\tau| \leq 4
\quad \text{and} \quad 
0.079 < \arg \tau < \pi -0.079. 
\]
\end{prop}
\begin{proof}
Let $x,y,z,w \in \pi_{1}(M)$ be as in the proof 
of Proposition~\ref{prop:modulus}. 
By taking conjugates for the Shimizu-Leutbecher lemma, 
we have 
if $x = \begin{pmatrix}
1 & 2 \\ 
0 & 1
\end{pmatrix}$ and 
$\begin{pmatrix}
a & b \\ 
c & d
\end{pmatrix}$ 
generate a discrete subgroup of $\mathrm{PSL}(2,\mathbb{C})$, 
then $c=0$ or $|c| \geq 1/2$. 
Now considering 
$y^{n}w^{m}= \begin{pmatrix}
1 & 0 \\ 
2m\tau +2n & 1
\end{pmatrix}$, 
we obtain 
$|m\tau +n| \geq 1/4$. 
Similarly, 
since 
$y = \begin{pmatrix}
1 & 0 \\ 
2 & 1
\end{pmatrix}$ and
$x^{n}z^{m}= \begin{pmatrix}
1 & 2(m/\tau) +2n \\ 
0 & 1
\end{pmatrix}$
generate a discrete subgroup, 
we have 
$|(m/\tau) +n| \geq 1/4$. 

\fig[width=12cm]{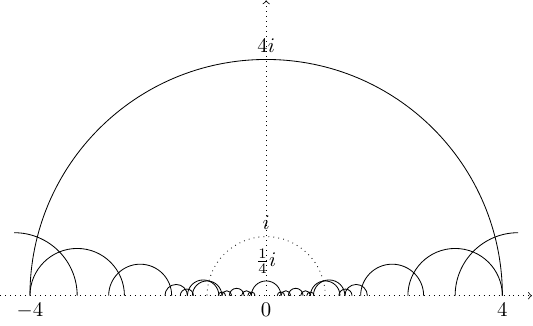}{A bound for the modulus $\tau$}

In fact, these conditions are equivalent to 
the inequalities for 
\[
(m,n)=(1,0), (4,\pm 1), (3,\pm 1), (2,\pm 1), (3,\pm 2), (4,\pm 3), (1, \pm 1). 
\]
Indeed, for large $(m,n)$, 
the equalities $|m \tau +n| = 1/4$ 
and $|(m/\tau) +n| = 1/4$ 
give small circles in $\mathbb{C}$. 
The above inequalities define 
the region bounded by 26 arcs as shown in Figure~\ref{fig:u3ps-region}. 
This region is 
symmetric about the imaginary axis 
and invariant under the inversion 
with respect to the unit circle. 
We recall that $\mathrm{Im} (\tau) >0$. 
The inequalities for $(m,n)=(1,0)$ imply that 
$\frac{1}{4} \leq |\tau| \leq 4$. 
A point of minimal slope of $\tau$ in the first quadrant 
satisfying these inequalities is $(93+\sqrt{55}i)/128$, 
which is contained in the intersection of two circles given by 
$|3\tau -2| = 1/4$ and $|4\tau -3| =1/4$. 
Therefore $\arg \tau \geq \arctan (\sqrt{55}/93) > 0.079$. 
\end{proof}

For the 3-punctured spheres $\Sigma_{1}, \dots , \Sigma_{n}$ 
of the type $A_{n}$, 
let $M_{n}$ be a regular neighborhood of the union 
of $\Sigma_{1} \cup \dots \cup \Sigma_{n}$ 
with the $n$ adjacent torus cusps. 
The frontier $\partial_{0} M_{n}$ is a 4-punctured sphere. 
For $\tau \in \mathbb{C} \setminus \{ 0 \}$, 
we define a representation 
$\rho_{\tau} \colon \pi_{1}(M_{n}) \to \mathrm{PSL}(2, \mathbb{C})$ 
whose restriction to $\pi_{1}(\Sigma_{i} \cup \Sigma_{i+1})$ 
for each $1 \leq i \leq n-1$ 
is conjugate to the one 
in the proof of Proposition~\ref{prop:modulus}. 
The parameter space $\mathbb{C} \setminus \{ 0 \}$ 
can be regarded as the character variety for $M_{n}$. 
We regard $\mathcal{C}_{n}$ 
as a subspace of $\mathbb{C} \setminus \{ 0 \}$. 
If $\tau \in \mathcal{C}_{n}$, 
then $\rho_{\tau}$ is the holonomy representation for $M_{n}$ 
with the cusp modulus $\tau$. 
The image of $\rho_{\tau}$ is discrete 
if and only if $\tau$ or $\bar{\tau} \in \mathcal{C}_{n}$. 

J\o rgensen's inequality~\cite{jorgensen1976discrete} implies that 
if representations of a group to $\mathrm{PSL}(2, \mathbb{C})$ 
have discrete images, 
their algebraic limits is elementary or has a discrete image. 
Hence $\mathcal{C}_{n}$ is closed in $\mathbb{C} \setminus \{ 0 \}$. 
We divide $\mathcal{C}_{n}$ into two subsets 
$\mathcal{C}_{n}^{\mathrm{incomp}} = \{ \tau \in \mathcal{C}_{n} \mid \rho_{\tau}$ 
is injective$\}$
and 
$\mathcal{C}_{n}^{\mathrm{comp}} = \{ \tau \in \mathcal{C}_{n} \mid \rho_{\tau}$ 
is not injective$\}$.

\begin{thm}
\label{thm:incomp}
The set $\mathcal{C}_{n}^{\mathrm{incomp}}$ 
is homeomorphic to a closed disk. 
\end{thm}

Theorem~\ref{thm:incomp} is due to Minsky~\cite{minsky1999classification}. 
Although in that paper he mainly considered once-punctured torus groups, 
he noted that the same argument holds for hyperbolic 3-manifolds with 
boundary consisting of 4-punctured spheres. 
This homeomorphism is described as follows. 
Let $\mathcal{T}_{0,4}$ denote 
the Teichm\"{u}ller space of a 4-punctured sphere. 
Its Thurston compactification 
$\overline{\mathcal{T}_{0,4}}$ is homeomorphic to a closed disk. 
The set $\mathcal{C}_{n}^{\mathrm{incomp}}$ is 
the image of a continuous map 
$\iota \colon \overline{\mathcal{T}_{0,4}} 
\to \mathbb{C} \setminus \{ 0 \}$. 
The Thurston boundary 
$\partial \mathcal{T}_{0,4} 
= \overline{\mathcal{T}_{0,4}} - \mathcal{T}_{0,4}$ 
can be identified with $\mathbb{R} \cup \{ \infty \}$. 
A rational points of $\partial \mathcal{T}_{0,4}$ 
corresponds to the homotopy class of an essential simple closed curve 
on the 4-punctured sphere. 
Here we conventionally regard $\infty$ as a rational point. 
A parameter $\tau \in \mathcal{C}_{n}^{\mathrm{incomp}}$ 
determines a hyperbolic structure of $M_{n}$ 
as $\mathbb{H}^{3} / \rho_{\tau}(\pi_{1}(M_{n}))$. 
If $\tau \in \mathrm{int} (\mathcal{C}_{n}^{\mathrm{incomp}})$, 
then the hyperbolic structure is geometrically finite, 
and the conformal structure on the infinite end $\partial_{0} M_{n}$ 
is given by $\iota^{-1} (\tau) \in \mathcal{T}_{0,4}$. 
If $\iota^{-1} (\tau)$ is 
a rational point of $\partial \mathcal{T}_{0,4}$, 
there is an annular cusp of the corresponding slope 
on $\partial_{0} M_{n}$. 
On the other hand, 
if $\iota^{-1} (\tau)$ is 
an irrational point of $\partial \mathcal{T}_{0,4}$, 
then $\mathbb{H}^{3} / \rho_{\tau}(\pi_{1}(M_{n}))$ is 
geometrically infinite. 

\begin{thm}
\label{thm:comp}
The set $\mathcal{C}_{n}^{\mathrm{comp}}$ 
is a countably infinite set. 
\end{thm}

If $\tau \in \mathcal{C}_{n}^{\mathrm{comp}}$, 
the manifold $\mathbb{H}^{3} / \rho_{\tau}(\pi_{1}(M_{n}))$ 
is the complement of a Montesinos link. 
We prepare the notion of Montesinos links. 

Let 
$B^{3} = \{ (x,y,z) \in \mathbb{R}^{3} | x^{2} + y^{2} + z^{2} \leq 1\}$. 
Consider the natural projection of $B^{3}$ into the $xy$-plane. 
A 2-tangle is two arcs properly embedded in $B^{3}$ 
whose endpoints are $\{ (\pm 1/\sqrt{2}, \pm 1/\sqrt{2}, 0) \}$. 
The equivalence of 2-tangles is given by 
isotopy fixing the endpoints of the arcs. 
A trivial tangle is a 2-tangle 
that is injected by the projection. 
The homotopy class of a (non-oriented) essential closed curve 
in the 4-punctured sphere 
$\partial_{0} B^{3} = 
\partial B^{3} - \{ (\pm 1/\sqrt{2}, \pm 1/\sqrt{2}, 0) \}$ 
is determined by a slope $r \in \mathbb{Q} \cup \{ \infty \}$. 
Here the slope $r=p/q$ for coprime integers $p$ and $q$ is 
defined so that 
the homology class of the loop is $\pm (pm+ql)$, 
where $m, l \in H_{1}(\partial_{0} B^{3}, \mathbb{Z})$ 
are respectively represented 
by the loops $\theta \in [0,2\pi] \mapsto 
(\sin \theta,0,\cos \theta), (0,\sin \theta,\cos \theta)$. 
A 2-tangle homeomorphic to a trivial tangle 
is determined up to isotopy by $r \in \mathbb{Q} \cup \{ \infty \}$, 
where the compressing disk for $\partial_{0} B^{3}$ 
in the complement has the boundary of slope $r$. 
This tangle is called a rational tangle of slope $r$. 

For $r_{1}, \dots , r_{n} \in \mathbb{Q} \cup \{ \infty \}$, 
the Montesinos link $L(r_{1}, \dots , r_{n})$ 
is defined by composing 
rational tangles of slopes $r_{1}, \dots , r_{n}$ 
as shown in Figure~\ref{fig:u3ps-montesinos}. 
We consider that the diagram is drawn in the $xy$-plane, 
which determines the slopes of tangles. 

\fig[width=10cm]{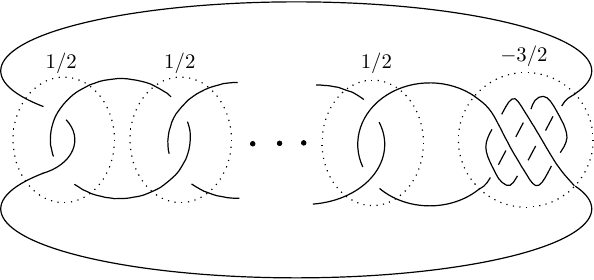}{The Montesinos link $L(1/2, 1/2, \dots , 1/2, -3/2)$}

\begin{proof}[Proof of Theorem~\ref{thm:comp}]
Let $\tau \in \mathcal{C}_{n}^{\mathrm{comp}}$. 
Equip $M_{n}$ with the metric for the cusp modulus $\tau$. 
Then $\partial_{0} M_{n}$ is compressible. 
By taking a regular neighborhood 
of the union of $M_{n}$ with the compressing disk, 
we obtain that the ambient hyperbolic 3-manifold 
$\mathbb{H}^{3} / \rho_{\tau}(\pi_{1}(M_{n}))$ is 
the union of $M_{n}$ with the complement of a trivial tangle. 
Hence the manifold $\mathbb{H}^{3} / \rho_{\tau}(\pi_{1}(M_{n}))$ 
is the complement of a Montesinos link 
$L(1/2, \dots , 1/2, r)$, 
where the number of tangles of slope 1/2 is $n+1$. 
The set $\mathcal{C}_{n}^{\mathrm{comp}}$ corresponds to 
the set of slopes $r$ 
such that $L(1/2, \dots , 1/2, r)$ are hyperbolic links. 

The classification of Montesinos links 
by Bonahon and Siebenmann~\cite{bonahon2010new} 
implies that 
all these links but the following exceptions are hyperbolic 
(see also \cite[Section 3.3]{futer2009angled}). 
\begin{itemize}
\item For $n=2$, the four slopes $r= -2, -3/2, -1, \infty$ are excluded. 
\item For $n=3$, the two slopes $r= -2, \infty$ are excluded. 
\item For $n \geq 4$, the slope $r= \infty$ is excluded. 
\end{itemize}
\end{proof}

\begin{rem}
\label{rem:comp}
If an annular cusp of an excluded slope is added to $M_{n}$, 
we obtain one of the manifolds 
$\mathbb{B}_{n+1}, \mathbb{T}_{3}$, and $\mathbb{T}_{4}$. 
If the union of a 3-punctured sphere 
with the compressing disk is an annulus, 
the filled manifold is not hyperbolic. 
\end{rem}

Let us consider the subset 
$\mathcal{C}_{n}^{\mathrm{fin}}$ of $\mathcal{C}_{n}$ 
consisting of the moduli 
that appears in a finite volume hyperbolic 3-manifold. 
Note that 
$\mathcal{C}_{n}^{\mathrm{comp}} \subset \mathcal{C}_{n}^{\mathrm{fin}}$. 
By Theorem~\ref{thm:brooks}, 
the restriction about finite volume 
does not give serious difference on bounds. 

\begin{thm}[Brooks~\cite{brooks1986circle}]
\label{thm:brooks}
Let $\Gamma < \mathrm{PSL}(2, \mathbb{C})$ 
be a geometrically finite Kleinian group. 
Then there exist arbitrarily small quasi-conformal deformations 
$\Gamma_{\epsilon}$ of $\Gamma$, 
such that $\Gamma_{\epsilon}$ is contained in 
the fundamental group of a finite volume hyperbolic 3-manifold. 
\end{thm}

We apply Theorem~\ref{thm:brooks} 
to $\Gamma = \rho_{\tau}(\pi_{1}(M_{n}))$. 

\begin{cor}
\label{cor:dense}
The set $\mathcal{C}_{n}^{\mathrm{fin}}$ is dense in $\mathcal{C}_{n}$. 
\end{cor}

\begin{prop}
\label{prop:bmodulus}
If 3-punctured spheres of the type $A_{n}$ 
are contained in the ones of the type $B_{2n}$, 
the adjacent cusp modulus $\tau$ is equal to $2i$. 
In particular, $2i \in \partial \mathcal{C}_{n}^{\mathrm{incomp}}$ 
for any $n \geq 2$. 
\end{prop}
\begin{proof}
The 3-punctured spheres of the type $B_{2n}$ are contained 
in the manifold $\mathbb{B}_{n+1}$, 
which is decomposed into $n+1$ regular ideal octahedra. 
This is obtained from the decomposition of $\mathbb{B}_{1}$ 
into a regular ideal octahedron shown in Figure~\ref{fig:u3ps-b1decomp}. 
We can construct a fundamental domain of each adjacent cusp by two Euclidean unit squares. 
Then the lengths of the meridian and longitude are respectively equal to 1 and 2. 
Hence the modulus $\tau$ is equal to $2i$. 

We have another proof by computing the representation. 
We use the elements 
\[
y = 
\begin{pmatrix}
1 & 0 \\ 
2 & 1
\end{pmatrix} \quad \text{and} \quad
z =
\begin{pmatrix}
1 & 2\tau^{-1} \\ 
0 & 1
\end{pmatrix}
\]
shown in Figure~\ref{fig:u3ps-modulus}. 
Suppose that $y$ and $z$ represents meridians. 
If 3-punctured spheres of the type $A_{2}$ 
are contained in the ones of the type $B_{4}$, 
the element 
\[
yzy^{-1}z^{-1} = 
\begin{pmatrix}
1-4\tau^{-1} & 8\tau^{-2} \\ 
-8\tau^{-1} & 1+4\tau^{-1}+16\tau^{-2}
\end{pmatrix}
\]
is parabolic. 
Hence $|\mathrm{tr}(yzy^{-1}z^{-1})| = |2+16\tau^{-2}| = 2$. 
Since $\mathrm{Im} (\tau ) >0$, 
we have $\tau = 2i$. 
\end{proof}

\begin{prop}
\label{prop:contain}
It holds that 
\[
\mathcal{C}_{n+1} \subset \mathcal{C}_{n}, \
\mathcal{C}_{n+1}^{\mathrm{incomp}} \subset \mathcal{C}_{n}^{\mathrm{incomp}}, \ \text{and} \
\mathcal{C}_{n+2} \subset \mathcal{C}_{n}^{\mathrm{incomp}}.
\] 
\end{prop}
\begin{proof}
The first two containments are obvious. 
If $\tau \in \mathcal{C}_{n}^{\mathrm{comp}}$, 
the number of the cusps of the ambient hyperbolic 3-manifold is $n+1$ or $n+2$ 
by the proof of Theorem~\ref{thm:comp}. 
Then $\tau \notin \mathcal{C}_{n+2}$. 
Therefore we have $\mathcal{C}_{n+2} \subset \mathcal{C}_{n}^{\mathrm{incomp}}$. 
\end{proof}

The sets $\mathcal{C}_{n}$ become arbitrarily smaller 
as $n$ increases. 

\begin{thm}
\label{thm:boundlim}
For $n \geq 2$, 
let $\tau_{n} \in \mathcal{C}_{n}$. 
Then $\lim\limits_{n \to \infty} \tau_{n} = 2i$. 
\end{thm}

By Corollary~\ref{cor:dense}, we may assume that 
$\tau_{n} \in \mathcal{C}_{n}^{\mathrm{fin}}$. 
Theorem~\ref{thm:boundlim} follows from 
Theorem~\ref{thm:filling} and Lemma~\ref{lem:fillb2n}. 
Theorem~\ref{thm:filling} states that 
a Dehn filling along a long slope gives 
a manifold with metric uniformly close to the original one 
in their thick parts. 
Here we use the normalized length of a slope, 
i.e. it is measured after rescaling the metric on the torus cusp 
to have unit area. 
We recall that 
$M_{[\epsilon, \infty)}$ is the $\epsilon$-thick part of a hyperbolic 3-manifold $M$. 
It is crucial that the estimates in Theorem~\ref{thm:filling} 
do not depend on manifolds.

\begin{thm}
\label{thm:filling}
For any $J>1$, there is a constant $K \geq 4\sqrt{2}\pi$ 
satisfying the following condition. 
Suppose that $M$ is obtained by a Dehn filling on a cusp 
of a finite volume hyperbolic 3-manifold $M_{0}$
along a slope of normalized length $L$ at least $K$. 
Then 
\begin{itemize}
\item[(i)]
the filled core loop 
(i.e. the core loop in the filled solid torus of the Dehn filling) 
is isotopic in $M$ to a closed geodesic 
of length $\leq 2\pi / (L^{2}-16\pi^{2})$, and 
\item[(ii)] 
there is a $J$-bilipschitz diffeomorphism 
$\varphi \colon (M_{0})_{[\epsilon, \infty)} \to M_{[\epsilon, \infty)}$ 
isotopic to 
the restriction of the natural inclusion $M_{0} \hookrightarrow M$. 
\end{itemize}
\end{thm}

The part (i) follows from the work of 
Hodgson and Kerckhoff~\cite{hodgson2005universal}. 
The part (ii) 
follows from the drilling theorem due to 
Brock and Bromberg~\cite{brock2004density}, 
which requires that the filled core loop is short. 
Magid~\cite[Section 4]{magid2012deformation} 
unified their arguments 
and gave the explicit bound in (i).

\begin{lem}
\label{lem:fillb2n}
Suppose that 
an orientable finite volume hyperbolic 3-manifold $M$ contains 
3-punctured spheres $\Sigma_{1}, \dots , \Sigma_{n}$ of the type $A_{n}$. 
Then there is a finite volume hyperbolic 3-manifold $M_{0}$ 
satisfying the following properties: 
\begin{itemize}
\item 
The manifold $M$ is obtained by a Dehn filling on a cusp of $M_{0}$ 
along a slope of normalized length at least $\sqrt{n+1}/2$. 
\item 
The 3-punctured spheres $\Sigma_{1}, \dots , \Sigma_{n}$ come from 
ones contained in 3-punctured spheres of the type $B_{2n}$ in $M_{0}$. 
\end{itemize}
\end{lem}
\begin{proof}
Let $\gamma$ be a loop in $M$ such that 
$\gamma$ and two meridians of an adjacent cusp 
bound a 2-punctured disk $\Sigma$ 
as shown in Figure~\ref{fig:u3ps-drill}. 
Let $N(\gamma)$ be an open regular neighborhood of $\gamma$. 
Assume that $M \setminus N(\gamma)$ contains an essential sphere $S$. 
Since $M$ is irreducible, the sphere $S$ bounds a ball in $M$. 
Note that the cusps of $M$ are incompressible. 
We apply the standard argument 
to reduce the intersection of surfaces in a 3-manifold 
by considering innermost intersection. 
Then by isotoping $S$ in $M \setminus N(\gamma)$, 
we may assume that 
$S$ is disjoint from $\Sigma_{1}, \dots , \Sigma_{n}$, 
and $S \cap \Sigma$ consists of 
(possibly empty) loops parallel to $\gamma$. 
If $S \cap \Sigma = \emptyset$, 
then $S$ bounds a ball in $M \setminus N(\gamma)$. 
Otherwise $\gamma$ bounds a disk $D$ in $M$ 
such that $\Sigma \cap D = \gamma$. 
Then $\Sigma \cup D$ is an essential annulus, 
which contradicts the fact that $M$ is hyperbolic. 
Hence $M \setminus N(\gamma)$ is irreducible. 

\fig[width=12cm]{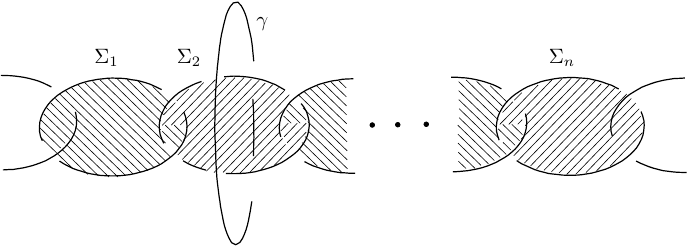}{A drilled loop $\gamma$ in $M$}

If $M \setminus N(\gamma)$ is hyperbolic, 
then it is sufficient to let $M_{0} = M \setminus N(\gamma)$. 
Suppose that $M \setminus N(\gamma)$ is not hyperbolic. 
Since $M$ is hyperbolic, 
$M \setminus N(\gamma)$ is not a graph 3-manifold. 
Hence there are non-empty family of essential disjoint tori $T_{1}, \dots , T_{m}$ 
for the JSJ decomposition of $M \setminus N(\gamma)$. 
Then we may assume that $T_{i} \cap \Sigma_{j} = \emptyset$. 
After the JSJ decomposition of $M \setminus N(\gamma)$ along $T_{1} , \dots , T_{m}$, 
there is a piece $M_{0}$ such that $\Sigma_{1} \cup \dots \cup \Sigma_{n} \subset M_{0}$. 
We may assume that the frontier of $M_{0}$ is $T_{1} , \dots , T_{k}$. 
Since $M$ is hyperbolic, 
each $T_{i}$ for $1 \leq i \leq k$ 
is compressible in $M$. 
Hence the manifold $M$ is obtained by a Dehn filling of $M_{0}$ along $T_{1} , \dots , T_{k}$. 
Since $\gamma$ is contained in a single solid torus for this Dehn filling, 
we have $k=1$. 
We can take $T_{1}$ by isotopy so that 
$\Sigma \cap M_{0}$ is a 3-punctured sphere in $M_{0}$. 
Hence $M_{0}$ contains 3-punctured spheres of the type $B_{2n}$. 
Moreover, the piece $M_{0}$ is hyperbolic.

The manifold $M_{0}$ can be decomposed along two 3-punctured spheres 
into the manifold $\mathbb{B}_{n+1}$ 
and a (possibly empty) manifold $M'_{0}$. 
We consider the cusp $C$ of $M_{0}$ on which we perform a Dehn filling. 
Let a meridian of $C$ be homotopic to $\gamma$ in $M$. 
Then the outer boundary component of a blue 3-punctured sphere 
of $\mathbb{B}_{n+1}$ 
in Figure~\ref{fig:u3ps-b2n} is a meridian. 
We can construct a fundamental domain of 
the annular cusp $C^{\prime}= C \cap \mathbb{B}_{n+1}$ 
by $4(n+1)$ Euclidean unit squares. 
This is obtained from the decomposition of $\mathbb{B}_{1}$ 
into a regular ideal octahedron shown in Figure~\ref{fig:u3ps-b1decomp}. 
By adding a fundamental domain $F$ of $C \setminus C^{\prime}$, 
we obtain a fundamental domain of $C$ 
as shown in Figure~\ref{fig:u3ps-domain}. 
Let $r \geq 0$ denote the height of $F$ 
with respect to the meridian as the base of $F$. 
Note that $M_{0}$ is 
$\mathbb{W}_{n+1}$ or $\mathbb{W}^{\prime}_{n+1}$
if $r=0$. 
When we normalize $C$ to have unit area, 
the length of the meridian is $2/\sqrt{n+1+r}$. 
Since $M$ is hyperbolic, 
the slope of the Dehn filling of $M_{0}$ is not the meridian. 
Hence the length of this slope is 
at least $\sqrt{n+1+r}/2$. 
\end{proof}

\fig[width=8cm]{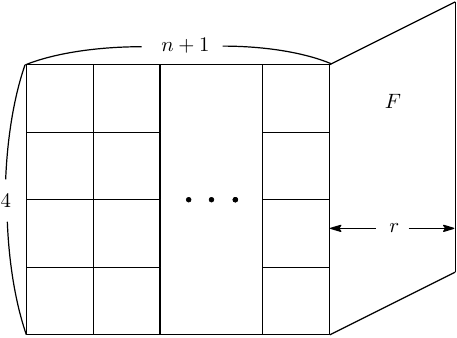}{A fundamental domain of the cusp $C$}

We finally complete the classification for infinitely many 3-punctured spheres. 
\begin{proof}[Proof of Theorem~\ref{thm:infinite}.]
If infinitely many intersecting 3-punctured spheres are placed linearly, 
the adjacent cusp modulus is equal to $2i$ by Theorem~\ref{thm:boundlim}. 
Then there is a cusp that bounds additional 3-punctured spheres. 
Hence the union of the 3-punctured spheres is $B_{\infty}$ or $\Whi_{\infty}$. 
The argument in Section~\ref{section:proof} implies that 
there is no other type of the union. 
\end{proof}

\section*{Acknowledgements} 
The author would like to thank Hidetoshi Masai 
for his several helpful comments. 
This work was supported by 
JSPS KAKENHI Grant Numbers 24224002, 15H05739.

\bibliography{ref-u3ps}

\textsc{Graduate School of Science and Engineering,  
Saitama University, 255 Shimo-Okubo, 
Sakura-ku, Saitama-shi, 338-8570, Japan.} 

\textit{E-mail address}: \texttt{kyoshida@mail.saitama-u.ac.jp}

\end{document}